\documentclass[11pt,a4]{article}

\usepackage[utf8]{inputenc}             % 
\usepackage[T1]{fontenc}                % Pour les accents
\usepackage{amssymb,amsmath,amsthm}
\usepackage{xspace}
\usepackage{graphicx}                   % Pour les images
\usepackage[french, english]{babel}            % Document en français
\usepackage{color}                      % mettre du texte en couleur
\usepackage[left=3cm, right=3cm]{geometry}
\usepackage[citecolor = blue]{hyperref}
\usepackage{dsfont}

\newcommand{\eps}{\varepsilon}
\newcommand{\Z}{\mathbb{Z}}
\newcommand{\R}{\mathbb{R}}

\newcommand{\N}{\mathbb{N}}
\newcommand{\E}{\mathbb{E}}
\newcommand{\prob}{\mathbb{P}}

\newcommand{\one}{\mathds{1}}
\newcommand{\man}{\mathcal{X}}

\newcommand{\calH}{\mathcal{H}}
\newcommand{\calK}{\mathcal{K}}

\newcommand{\calC}{\mathcal{C}}
\newcommand{\Hom}{\text{\textup{Hom}}}
\newcommand{\Hess}{\textup{Hess}}

\newcommand{\sgn}{\textup{sgn}}
\newcommand{\vol}{\textup{Vol}}

\newenvironment{prf}[1][]
{\vskip 2mm  {\it \bf Proof#1. }}{$\Box$ \vskip 2mm}

\newtheorem{theorem}{Theorem}
\numberwithin{theorem}{section}
\newtheorem{lemma}[theorem]{Lemma}

\newtheorem{definition}[theorem]{Definition}
\newtheorem{remark}{Remark}
\newtheorem{condition}{Condition}
\newtheorem{claim}{Claim}

\title{Expected number of nodal components for cut-off fractional Gaussian fields}
\author{Alejandro Rivera}
\date{\today}

\begin{document}
\maketitle

\vspace{2cm}

\begin{abstract}
Let $(\man,g)$ be a closed Riemmanian manifold of dimension $n>0$. Let $\Delta$ be the Laplacian on $\man$, and let $(e_k)_k$ be an $L^2$-orthonormal and dense family of Laplace eigenfunctions with respective eigenvalues  $(\lambda_k)_k$. We assume that $(\lambda_k)_k$ is non-decreasing and that the $e_k$ are real-valued. Let $(\xi_k)_k$ be a sequence of iid $\mathcal{N}(0,1)$ random variables. For each $L>0$ and $s\in\R$, possibly negative, set

\[
f^s_L=\sum_{0<\lambda_j\leq L}\lambda_j^{-\frac{s}{2}}\xi_je_j\, .
\]

Then, $f_L^s$ is almost surely regular on its zero set. Let $N_L$ be the number of connected components of its zero set. If $s<\frac{n}{2}$, then we deduce that there exists $\nu=\nu(n,s)>0$ such that $N_L\sim \nu \vol_g(\man)L^{n/2}$ in $L^1$ and almost surely. In particular, $\E[N_L]\asymp L^{n/2}$. On the other hand, we prove that if $s=\frac{n}{2}$ then
\[
\E[N_L]\asymp \frac{L^{n/2}}{\sqrt{\ln\left(L^{1/2}\right)}}\, .
\]
In the latter case, we also obtain an upper bound for the expected Euler characteristic of the zero set of $f_L^s$ and for its Betti numbers. In the case $s>n/2$, the pointwise variance of $f_L^s$ converges so it is not expected to have universal behavior as $L\rightarrow+\infty$.
\end{abstract}

\vfill

\textsc{MSC2010 subject classification: 60G15, 34L20, 60G05,  	60G22}

\pagebreak

\tableofcontents

\section{Introduction}\label{s.introduction}

\subsection{Setting and main results}\label{ss.results}

In this manuscript we study the number of \textbf{nodal components} of random linear combinations of eigenfunctions of the Laplacian on a closed Riemmanian manifold, that is, the number of connected components of the zero set of such random functions. The study of such components goes back to~\cite{ns09} where the authors consider random eigenfunctions with eigenvalue $L$ of the Laplacian on the sphere $S^2$ and prove that the number of components concentrates around $cL$ for some $c\in]0,+\infty[$ as $L\rightarrow+\infty$. Later, in~\cite{ns15}, a similar result\footnote{Note however that the concentration rate is much weaker than in the former setting~\cite{ns09}.} was proved regarding the number of components of general Gaussian fields on Riemmanian manifolds. Meanwhile, using different methods, in~\cite{gawe14,gawe15}, the authors determined the rate of growth of the Betti numbers of the nodal set for a particular model of random linear combination of eigenfunctions on a Riemmanian manifold. Most of the arguments in these two papers were quite general but required inputs from spectral analysis at a few key steps in the proof. More precisely, the authors relied on a result from semi-classical analysis (see Theorem 2.3 of~\cite{gawe14} in which the authors extend a result from~\cite{ho68}). Other works in this field are~\cite{sawi,bewi16,kuwi17,casa17}. All of the aforementioned works study parametric families of smooth functions $(f_L)_{L\geq 0}$ on a manifold of dimension $n$ that vary at a natural scale $L^{-1/2}$ and that posess 'local limits'. In the case of~\cite{ns15} this is an explicit assumption while in the other cases, it follows from results about spectral asymptotics. As a result, the number of connected components in a fixed compact set is of order $L^{n/2}$. In contrast, the recent~\cite{ale161} introduced a model of random linear combinations of eigenfunctions of the Laplacian on a closed surface that did not have 'local limits'. A natural problem is to determine the rate of growth of the number of nodal domains for this new model. The present work is set in the continuation of the articles mentioned above and provides an answer to this question.\\
We consider a smooth compact Riemmanian manifold $(\man,g)$ of dimension $n>0$ with no boundary. Let $|dV_g|$ be the Riemmanian density and $\Delta$ the Laplacian operator induced by $g$ on $\man$. Let $g^{-1}$ be the metric induced on $T^*\man$ by $g$. Since $\man$ is closed, the spectrum of $\Delta$ is discrete and made up of a sequence of non-negative eigenvalues $(\lambda_k)_{k\geq 0}$. Moreover, on can find a sequence of corresponding eigenfunctions $(e_k)_{k\in\N}$ (i.e., $\Delta e_k=\lambda_k e_k$ for all $k$) that are real valued, smooth and normalized so as to form a Hilbert basis for $L^2(\man,|dVg|)$. Let $(\xi_k)_{k\geq 0}$ be a sequence of independent centered Gaussians of unit variance and, for each $s\in\R$ (possibly negative) and $L>0$, set

\begin{equation}\label{e.field}
f_L^s=\sum_{0<\lambda_j\leq L}\lambda_j^{-\frac{s}{2}}\xi_j e_j\, .
\end{equation}

This formula defines a smooth Gaussian field on $\man$ which we call \textbf{cut-off fractional Gaussian field} because of the cut-off $\lambda_j\leq L$ and the fractional power $-\frac s 2$. A simple calculation shows that the covariance function for $f_L^s$ is

\[
\E[f_L^s(x)f_L^s(y)]:=K^s_L(x,y)=\sum_{0<\lambda_j\leq L}\lambda_j^{-s} e_j(x)e_j(y)\, .
\]

The behavior of $K^s_L$ near the diagonal as $L\rightarrow +\infty$ was studied in~\cite{riv_weyl} for $s\leq n/2$. It is well known, at least for $s=0$, that for $L$ large enough, the nodal set $Z_L=\{x\in\man\, |\, f^s_L(x)=0\}$ is almost surely smooth (see Lemma \ref{l.as_smooth} below). We are interested in $N_L$, the number of connected components of the nodal set. In the case where $s<n/2$, combining results from~\cite{riv_weyl} and \cite{ns15}, we deduce the following result:

\begin{theorem}\label{t.supercrit}
Suppose that $s<n/2$. Then, there exists a (deterministic) constant $\nu_{n,s}>0$ depending only on $s$ and $n$ such that $L^{-n/2}N_L$ converges to $\nu_{n,s} \vol_g(\man)$ in $L^1$ as $L\rightarrow+\infty$. 
\end{theorem}

On the other hand, in the case where $s=n/2$, the asymptotic behavior of the field is quite different (see Theorem 1.2 of \cite{riv_weyl} or Theorem \ref{t.covariance}). In particular, it does not have non-trivial local limits. In this case, we prove the following result:

\begin{theorem}\label{t.main}
Suppose that $s=n/2$. Then, there exist constants $0<c<C<+\infty$ where $C$ depends only on $n$ such that for $L$ large enough,
\[
c\vol_g(\man)\frac{L^{n/2}}{\sqrt{\ln\left(L^{1/2}\right)}}\leq \E[N_L]\leq C\vol_g(\man)\frac{L^{n/2}}{\sqrt{\ln\left(L^{1/2}\right)}}\, .
\]
Moreover, there exists $\rho=\rho(\man)<+\infty$ such that if $N_L(\rho)$ is the number of connected components with diameter at most $\rho L^{-1/2}$ then
\[
\E\left[N_L(\rho)\right]\geq c\vol_g(\man)\frac{L^{n/2}}{\sqrt{\ln\left(L^{1/2}\right)}}\, .
\]
\end{theorem}
\begin{remark}
We choose to keep the $L^{1/2}$ inside the logarithm in the statement ot Theorem \ref{t.main} and below because $L^{-1/2}$ is the correct scaling for the typical variation length of the field. Thus, $L^{-1/2}$ is a more natural quantity for our results than $L$.
\end{remark}
\begin{remark}
The upper constant is explicit, as explained in Theorem \ref{t.betti} below.
\end{remark}
\begin{remark}
In the case where $s=1$, the field $f_L^s$ is the cut-off Gaussian Free Field introduced in \cite{ale161}. By construction, this field converges in distribution to the Gaussian Free Field (see \cite{shef_gff}). In this case, Theorem \ref{t.main} implies that in dimension $n=2$, $\E\left[N_L\right]\asymp \frac{L}{\sqrt{\ln(L)}}$ while in dimension $n\geq 3$, $\E\left[N_L\right]\asymp L^{n/2}$.
\end{remark}
\begin{remark}
In the case where $s> n/2$, the results of \cite{riv_weyl} do not give asymptotic results for the pointwise variance of the field. This is because, by Weyl's law, $\lambda_k\asymp k^{2/n}$ so the variance of $f^s_L(x)$, which is, $\sum_{0<\lambda_j\leq L} \lambda_j^{-s}$, converges as $L\rightarrow+\infty$. In particular, it depends on the geometry of $\man$.
\end{remark}
For the lower bound, the most common strategy is to construct a 'barrier' (see for instance Claim 3.2 of \cite{ns09} or Corollary 1.11 of \cite{gawe15}). This amounts to constructing a model function with a nodal component inside a given ball and proving that the field does not deviate too much from this model with positive probability. In this case, the barrier cannot hold with probability bounded from below without contradicting the upper bound. However, by pinning the field near zero at a given point, we manage to construct such a barrier losing only a factor of $\sqrt{\ln\left(L^{1/2}\right)}$. For the upper bound, we follow the approach of \cite{gawe14}. Indeed, although their main result doesn't apply, their strategy still does. The strategy of \cite{gawe14} is to count the critical points of a Morse function on the nodal set $Z_L=\{x\in\man\ :\ f_L^s(x)=0\}$. This provides not only an upper bound on the number of nodal components but also the following result:

\begin{theorem}\label{t.betti}
Assume that $s=n/2$. For each $L\geq 1$ let $\chi(Z_L)$ be the Euler characteristic of $Z_L$ and for each $i\in\{0,\dots,n-1\}$ let $b_i(Z_L)$ be the i-th Betti number of $Z_L$. As $L\rightarrow +\infty$,
\[
\E[\chi(Z_L)]= o\left(\frac{L^{n/2}}{\sqrt{\ln\left(L^{1/2}\right)}}\right)\, .
\]
Moreover, for each $i\in\{0,\dots,n-1\}$, if $b_i(Z_L)=\textup{rank}_{\Z}H_i(Z_L;\Z)$ is the $i$-th Betti number of $Z_L$,
\[
\limsup_{L\rightarrow +\infty}\frac{\sqrt{\ln\left(L^{1/2}\right)}}{L^{n/2}}\E[b_i(Z_L)]\leq A_n^i \vol_g(\man)
\]
where $A_n^i$ is defined as follows. Let $M$ be a centered Gaussian vector with values in the space of symmetric matices $\textup{Sym}_{n-1}(\R)$ satisfying for any $i,j,k,l\in\{1,\dots,n-1\}$ such that $i<j$, $k\leq k$ and $(i,j)\neq (k,l)$, $\E\left[M_{ii}M_{jj}\right]=\E\left[M_{ij}^2\right]=1$, $\E\left[M_{ii}^2\right]=3$ and $\E\left[M_{ij}M_{kl}\right]=0$. Then,
\[
A_n^i=\frac{\E[|\det(M)|\one[\sgn(M)=i]]}{\sqrt{\pi^{n+1}2^{2n-1}n(n+2)^{n-1}}}\, .
\]
Here $\sgn$ denotes the number of negative eigenvalues of the symmetric matrix $M$.
\end{theorem}

In particular, for the upper bound in Theorem \ref{t.main} one can set $C=A_n^0$.

\subsection{Proof strategy}\label{ss.strategy}

The crucial tool for the proof of Theorems \ref{t.supercrit} and \ref{t.main} is the following result from~\cite{riv_weyl}. It provides an estimate for the covariance of $K^s_L$. Recall that $g^{-1}$ is the metric induced on $T^*\man$ by $g$. Given $(w,\xi)\in T^*\man$, we write $|\xi|_w^2:=g_w^{-1}(\xi,\xi)$.
 
\begin{theorem}[Corollaries 1.4 and 1.5 of~\cite{riv_weyl}]\label{t.covariance}
$ $
\begin{enumerate}
\item Assume that $s<n/2$. Fix $x_0\in\man$ and consider local coordinates $x=(x_1,\dots,x_n)$ centered at $x$ such that $|dV_g|$ agrees the Lebesgue measure in these coordinates. Then, there exists $U\subset\R^n$ a neighborhood of $0$ such that for each $\alpha,\beta\in\N^n$, uniformly for $w\in U$ and $x,y$ in compact subsets of $\R^n$,
\[
\lim_{L\rightarrow+\infty}L^{s-n/2} \partial^\alpha_x\partial^\beta_yK^s_L(w+L^{-1/2}x,w+L^{-1/2}y)=\frac{1}{(2\pi)^n}\int_{|\xi|_w^2\leq 1}e^{i\langle\xi,x-y\rangle}\frac{(i\xi)^\alpha(-i\xi)^\beta}{|\xi|^{2s}}d\xi
\]
where for any $\gamma\in\N^n$ we set $|\gamma|=\gamma_1+\gamma_2+\dots+\gamma_n$.
\item Assume that $s=n/2$. Fix $x_0$ and a set of local coordinates centered at $x_0$ such that the density $|dV_g|$ in these coordinates agreees with the Lebesgue measure. Then, there exists $U\subset\R^n$ a neighborhood of $0$ such that uniformly for $x,y\in U$ and $L\geq 1$

\[
K^s_L(x,y)=\frac{|S^{n-1}|}{(2\pi)^n}\left(\ln\left(L^{1/2}\right)-\ln_+\left(L^{1/2}|x-y|\right)\right)+O(1)\, .
\]
Here $|S^{n-1}|$ is the area of the Euclidean unit sphere in $\R^n$ and $\ln_+(t):=\ln(t)\vee 0$.
\item Assume that $s=n/2$. Fix $x_0\in\man$ and consider local coordinates $x=(x_1,\dots,x_n)$ centered at $x$ such that $|dV_g|$ agrees the Lebesgue measure in these coordinates. Then, there exists $U\subset\R^n$ a neighborhood of $0$ such that for each $\alpha,\beta\in\N^n$ such that $(\alpha,\beta)\neq 0$, uniformly for $w\in U$ and $L\geq 1$,
\[
\lim_{L\rightarrow+\infty}L^{-(|\alpha|+|\beta|)/2}\partial^\alpha_x\partial^\beta_yK^s_L(x,y)|_{x=y=w}=\frac{1}{(2\pi)^n}\int_{|\xi|_w^2\leq 1}\frac{(i\xi)^\alpha(-i\xi)^\beta}{|\xi|^n}d\xi\, .
\]
This estimate is the same as 1. restricted to the diagonal as long as $(\alpha,\beta)\neq 0$.
\end{enumerate}
\end{theorem}

The original result is somewhat more general and in particular provides an estimate of the error terms in each case but since our results are bounds up to a constant factor, these will not be of use to us.\\

As we shall see in Section \ref{s.proofs}, this result implies that for $s<n/2$, the family $(f^s_L)_{L\geq 1}$ satisfies the assumptions for the main result of \cite{ns15} which directly implies Theorem \ref{t.supercrit}
.\\

We prove the upper and lower bounds Theorem \ref{t.main} in two separate sections. For the upper bound, we follow the strategy of \cite{gawe14}. More explicitely, we fix a function $p\in C^\infty(\man)$ with at most a countable number of critical points\footnote{Such functions always exist since for instance Morse functions are dense in $C^\infty(\man)$.}. For each $i\in\{0,\dots,n-1\}$, each $L\geq 1$ and each Borel subset $B\subset\man$, let $m_i(p,f_L,B)$ be the number of critical points of $p|_{Z_L}$ of index\footnote{Recall that the index of a critical point of a given function is the number of negative eigenvalues of the Hessian of this function at the critical point.} $i$. In Section \ref{s.upper_bound} we prove the following theorem.

\begin{theorem}\label{t.morse_rice}
Let $M$ be a centered Gaussian vector with values in the space of symmetric matices $\textup{Sym}_{n-1}(\R)$ satisfying for any $i,j,k,l\in\{1,\dots,n-1\}$ such that $i<j$, $k\leq l$ and $(i,j)\neq (k,l)$, $\E\left[M_{ii}M_{jj}\right]=\E\left[M_{ij}^2\right]=1$, $\E\left[M_{ii}^2\right]=3$ and $\E\left[M_{ij}M_{kl}\right]=0$. Then, for $L$ large enough, $Z_L$ is almost surely smooth, $p|_{Z_L}$ is almost surely Morse and for any Borel subset $B\subset\man$, as $L\rightarrow +\infty$,
\begin{equation}\label{e.morse_rice}
\E\left[m_i(p,f_L,B)\right]\sim C_n\vol_g(B)\E\left[\left|\det\left(M\right)\right|\one[\sgn\left(M\right)=i]\right]\frac{L}{\sqrt{\ln\left(L^{1/2}\right)}}
\end{equation}
where
\[
C_n=\frac{1}{\sqrt{\pi^{n+1}2^{2n-1}n(n+2)^{n-1}}}
\]
and where $sgn(M)$ is the number of negative eigenvalues of the matrix $M$.
\end{theorem}

Theorem \ref{t.main} as well as Theorem \ref{t.betti} will then follow from the Morse inequalities. For the lower bound, we prove that given a ball of radius $\asymp L^{-1/2}$, the probability that this ball contains a nodal component is bounded from below by a constant multiple of $\left(\ln\left(L^{1/2}\right)\right)^{-1/2}$. This result actually holds for log-correlated Gaussian fields with only Hölder regularity. More precisely, in Section \ref{s.lower_bound} we prove the following result.

\begin{theorem}\label{t.tent}
Fix $n\in\N$ and $U\subset\R^n$ an open subset containing $\overline{B(0,1)}$. Let $(f_\lambda)_{\lambda\geq 1}$ be a family of continuous centered Gaussian fields on $U$ satisfying the following properties.

\begin{enumerate}
\item There exists $a<+\infty$ for each $x,y\in U$ and $\lambda\geq 1$,
\[
\left|\E\left[f_\lambda(x)f_\lambda(y)\right]-\ln(\lambda)+\ln_+\right(\lambda|x-y|\left)\right|\leq a\, .
\]
\item There exists $\alpha\in]0,2]$ and $b<+\infty$ such that for each $x,y\in U$ and each $\lambda\geq 1$ satisfying $\lambda |x-y|\leq 1$,
\[
\E\left[(f_\lambda(x)-f_\lambda(y))^2\right]\leq b^2\lambda^\alpha|x-y|^\alpha\, .
\]
\end{enumerate}

There exist $\rho=\rho(a,b,\alpha,n)>0$, $\kappa=\kappa(a,b,\alpha,n)>0$ and $\lambda_0=\lambda_0(a,b,\alpha,n)\in [1,+\infty[$ such that the following holds. Let $\calC^\lambda$ be the event that $f_\lambda^{-1}(0)$ has a connected component included in the ball $B(0,\rho/\lambda)$. Then, for each $\lambda\geq \lambda_0$.
\[
\prob\left[\calC^\lambda\right]\geq \kappa\ln(\lambda)^{-1/2}\, .
\]
\end{theorem}

Theorem \ref{t.tent} plays the same role as the 'barrier lemma' (Claim 3.2) of \cite{ns09} or as Theorem 0.3 of \cite{gawe15}. However, in this setting, we do not (and cannot!) obtain a uniform lower bound on the probability of having a nodal component inside a given small ball. Moreover, the behavior of log-correlated random fields is quite different from that of locally-translation-invariant random fields. Indeed, just as the aforementioned results were used to obtain lower bounds on the expected number (and topology) of connected components of the nodal set, Theorem \ref{t.main} will follow by packing $\man$ with disjoint small balls of radius $\asymp\lambda^{-1}=L^{-1/2}$ and adding up the expected nodal components contained in each ball. The proof of Theorem \ref{t.tent} combines tools (see Lemma \ref{l.tail_of_the_supremum}) from the theory of smooth Gaussian fields with the FKG inequality (see Lemma \ref{l.FKG}) from statistical mechanics.

\paragraph{Acknowledgements.}$ $\\

I am thankful to my advisor Damien Gayet for suggesting I study this question and for his help in the presentation of this work.\\

I am also grateful to the referee for the very thorough and helpful report I received in the review process.

\section{Proof of the main results}\label{s.proofs}

The object of this section is to prove Theorems \ref{t.supercrit}, \ref{t.main} and \ref{t.betti} using the results presented in Subsection \ref{ss.strategy} as well as Theorem 3 of \cite{ns15}. We start with the proof of Theorem \ref{t.supercrit}. As explained above, we simply check that Theorem \ref{t.covariance} implies that the field satisfies the assumptions of Theorem 3 of \cite{ns15}.

\begin{proof}[Proof of Theorem \ref{t.supercrit}]
Fix $2s<n$. By the first point of Theorem \ref{t.covariance}, for any $x_0\in\man$ and any set of local coordinates on a chart $U\subset\man$ centered at $x_0$ and pushing $|dV_g|$ to the Lebesgue measure, there exists a neighborhood $V$ of $0$ in $U$ such that we have
\[
\lim_{L\rightarrow+\infty}L^{s-n/2}K_L^s\Big(w+L^{-1/2}x,w+L^{-1/2}y\Big):=\frac{1}{(2\pi)^n}\int_{|\xi|_w^2\leq 1}e^{i\langle \xi,x-y\rangle}|\xi|^{-2s}d\xi
\]
where the convergence takes place in $C^\infty$ with respect to $(x,y)\in V\times V$, uniformly with respect to $w\in V$. This shows that the kernels $K_L^s$ have, in the terminology of \cite{ns15}, translation-invariant limits on $V$ (see Definition 2 of \cite{ns15})\footnote{The parameter $L$ used in \cite{ns15} corresponds to the quantity $L^{1/2}$ of the present work.} and that it satisfies the norm estimates required for the parametric Gaussian ensemble (see Definition 1 of \cite{ns15}) to be locally uniformly controllable (see Definition 4 of \cite{ns15}) on $V$. The spectral measure at $w$ equals
\[
\rho_w(\xi)=|\xi|_w^{-2s}\one\left[|\xi|_w^2\leq 1\right]d\xi
\]
so it has no atoms and for each $i,j\in\{1,\cdots,n\}$, by parity,
\begin{align*}
\lim_{L\rightarrow+\infty}\partial_{x_i}\partial_{y_j}L^{s-(n+2)/2}K_L^s(x,y)|_{y=x=w}&=-\frac{1}{(2\pi)^n}\int_{|\xi|_w^2\leq 1} \xi_i\xi_j |\xi|_w^{-2s}d\xi\\
&=-\frac{\delta_{ij}}{n(2\pi)^n}\int_{|\xi|_w^2\leq 1}|\xi|_w^{2-2s}d\xi>0\, .
\end{align*}
In particular, the ensemble is locally uniformly non-degenerate (see Definition 3 of \cite{ns15}) and, together with the previous estimate on $K_L^s$, it is locally uniformly controllable. Since, finally, the spectral measure has no atoms, the Gaussian ensemble $\left(L^{s/2-n/4}f_L\right)_{L>0}$, in local coordinates, defines a tame ensemble (see Definition 5 of \cite{ns15}). Moreover, since such charts exist around each $x_0\in\man$, by  the criterion given in Subsection 1.4.1 of \cite{ns15}, the sequence $\left(L^{s/2-n/4}f_L\right)_{L>0}$ forms a tame parametric Gaussian ensemble on $\man$ (see Definition 6 of \cite{ns15}). Thus, by Theorem 3 of \cite{ns15} and the remark 1.5.2 that follows it, there exists a locally finite (and therefore finite since $\man$ is compact) Borel measure $\mathfrak{n}_\infty$ on $\man$ such that the sequence $\left(L^{n/2}N_L\right)_L$ converges in $L^1$ to $\mathfrak{n}_\infty(\man)<+\infty$.  Recall that the parameter $L$ in Nazarov and Sodin's theorem corresponds to $L^{1/2}$ of the present work. All that remains is to show that this quantity is positive. Moreover, according to the third item of Theorem 3 of \cite{ns15}, the density of $\mathfrak{n}_\infty$ with respect to $|dV_g|$ at a point $x\in\man$ is given by the constant $\nu$ given by item $(\rho_3)$ of Theorem \cite{ns15} for the limiting ensemble at the point $x$. Note that we may always require that the measure-preserving coordinates around $x_0$ be isometric at $x_0$ (i.e., we assume that the differential at $x_0$ of the diffeomorphism defining the local coordinates is an isometry from $(T_{x_0}\man,g_{x_0})$ to $\R^n$ equipped with the Euclidean scalar product). If so, the limiting spectral measure depends only on $n$ and $s$. Thus, $\nu=\nu_{n,s}$ is constant that depends only on $n,s$ and $\mathfrak{n}_\infty=\nu_{n,s}|dV_g|$.  Moreover, since the support of the spectral measure contains $0$, it satisfies Pjetro Majer's interior point criterion (see Appendix C.2 of \cite{ns15}) which in turn implies condition $(\rho_4)$ for Theorem 1 of \cite{ns15}. This shows that $\nu_{n,s}>0$ and so $\mathfrak{n}_\infty(\man)=\nu_{n,s}\vol_g(\man)>0$.
\end{proof}

Next we check Theorem \ref{t.betti}.

\begin{proof}[Proof of Theorem \ref{t.betti}]
Fix $p\in C^\infty(\man)$. By Theorem \ref{t.morse_rice}, $p|_{Z_L}$ is almost surely a Morse function. By the (weak) Morse inequalities (see Theorem 5.2 of \cite{milnor}), we have first for each $i\in\{0,\dots,n-1\}$,
\[
b_i\left(Z_L\right)\leq m_i(p,f_L,\man)
\]
where the $m_i$ are as in Theorem \ref{t.morse_rice}. Taking expectations, by Equation \eqref{e.morse_rice}, we obtain the upper bound on $\E\left[b_i(Z_L)\right]$ announced in Theorem \ref{t.betti}. Theorem 5.2 of \cite{milnor}  also implies that
\[
\chi(Z_L)=\sum_{i=0}^{n-1} m_i(p,f_L,\man)\, .
\]
Taking expectations and using Equation \eqref{e.morse_rice} on the right-hand side, we get, as $L\rightarrow +\infty$:
\begin{equation}\label{e.betti.pf.1}
\E\left[\chi(Z_L)\right]= C_n\vol_g(\man)\times \sum_{i=0}^{n-1}\E\left[|\det(M)|\one[\sgn(M)=i]\right]\frac{L}{\sqrt{\ln\left(L^{1/2}\right)}}+o\left(\frac{L}{\sqrt{\ln\left(L^{1/2}\right)}}\right)
\end{equation}
where $C_n=\frac{1}{\sqrt{\pi^{n+1}2^{2n-1}n(n+2)^{n-1}}}$ and $M$ is as described in the statement of Theorem \ref{t.betti}. Observe now that since $M$ is a symmetric matrix, then
\[
\sum_{i=1}^{n-1}(-1)^i|\det(M)|\one[\sgn(M)=i]=\det(M)\, .
\]
Therefore,
\begin{equation}\label{e.betti.pf.2}
\sum_{i=0}^{n-1}\E\left[|\det(M)|\one[\sgn(M)=i]\right]=\E[\det(M)]\, .
\end{equation}
Let us compute $\E[\det(M)]$. We claim the following:
\begin{claim}\label{cl.determinant_computation}
Let $M$ be a centered Gaussian vector in $\textup{Sym}_m(\R)$ the space of $m\times m$ symmetric matrices, with the following covariance structure: for any $i,j,k,l\in\{1,\dots,m\}$ such that $i<j$, $k\leq l$ and $(i,j)\neq (k,l)$, $\E[M_{ii}M_{jj}]=\E[M_{ij}^2]=1$, $\E[M_{ii}^2]=a\geq 1$ and $\E[M_{ij}M_{kl}]=0$. Then,
\[
\E[\det(M)]=0\, .
\]
\end{claim}
\begin{proof}
Since $M$ is centered, $M$ has the same law as $-M$. If $m$ is odd, this implies that $\det(M)$ has the same law as $-\det(M)$ so $\E[\det(M)]=0$. Let us now assume that $m=2q$ with $q\in\N$. We have
\begin{equation}\label{e.determinant_expectation.1}
\E[\det(M)]=\sum_{\sigma\in\mathfrak{S}_m}\textup{sgn}(\sigma)\E\left[\prod_{i=1}^m M_{i\sigma(i)}\right]
\end{equation}
where $\mathfrak{S}_m$ is the set of permutations of $\{1,\dots,m\}$ and $\textup{sgn}(\sigma)$ is the signature of the permutation $\sigma$. Let us fix $\sigma\in\mathfrak{S}_m$ and compute its corresponding term in the sum \eqref{e.determinant_expectation.1}. Let $i\in\{1,\dots,m\}$ be such that $\sigma(i)\neq i$. Then, $M_{i\sigma(i)}$ is independent from all the other factors in the product except if the factor $M_{\sigma(i)i}$ appears. In other words, $\sigma(\sigma(i))=i$. Thus, if $\sigma$ has cycles of order greater than two, its corresponding term in \eqref{e.determinant_expectation.1} vanishes. Let us assume that $\sigma$ only has cycles of order either one or two and let $k\in\{0,\dots,q\}$ be the number of cycles of order two of $\sigma$. Then, $\sigma$ is conjugated to the product of transpositions $(12)\dots(2k-1)(2k)$ (with the convention that $\sigma=id$ if $k=0$. Moreover, since the law of $M$ is invariant under conjugation by permutation matrices, we have, by independence of the off-diagonal coefficients
\[
\E\left[\prod_{i=1}^mM_{i\sigma(i)}\right]=\E\left[\prod_{i=k+1}^qM_{ii}\right]\times \prod_{i=1}^k\E\left[M_{i(i+1)}^2\right]=\E\left[\prod_{i=1}^{2(q-k)}M_{ii}\right]\, .
\]
To compute the expectation of the product, we introduce $Y_1,\dots,Y_{2(q-k)}$ i.i.d centered normals with variance $a-1$ and $Y$ an independent centered normal with unit variance. Then, the vectors $(M_{ii})_{1\leq i\leq 2(q-k)}$ and $(Y_i+Y)_{1\leq i\leq 2(q-k)}$ have the same law so
\[
\E\left[\prod_{i=1}^{2(q-k)}M_{ii}\right]=\E\left[\prod_{i=1}^{2(q-k)}(Y+Y_i)\right]\, .
\]
Developing the sum, by independence of the $Y_i$, the only term with non-vanishing expectation is $E\left[Y^{2(q-k)}\right]=(2(q-k))!!$. Here, $(2(q-k))!!$ denotes the product of all the odd integers between $0$ and $2(q-k)$. Thus,
\begin{equation}\label{e.determinant_expectation.2}
\E\left[\prod_{i=1}^mM_{i\sigma(i)}\right]=(2(q-k))!!\, .
\end{equation}
Now, the number of permutations with $k$ cycles of order two and no cycles of higher order is
\begin{equation}\label{e.determinant_expectation.3}
\binom{2q}{2k}(2k)!!\, .
\end{equation}
Indeed, the factor $\binom{2q}{2k}$ encodes the choice of the support of $\sigma$. Having fixed the support of $\sigma$, there are $2k-1$ choices for the image of the first element. Extracting the resulting two-cycle from the support leaves $2k-2$ elements. Equation \eqref{e.determinant_expectation.3} is then established by induction. Combining \eqref{e.determinant_expectation.1} with \eqref{e.determinant_expectation.2} and \eqref{e.determinant_expectation.3}, we get
\begin{align*}
\E[\det(M)]&=\sum_{k=0}^q(-1)^k\binom{2q}{2k}(2k)!!(2(q-k))!!\\
&=(2q)!\sum_{k=0}^q(-1)^k\frac{(2k)!!}{(2k)!}\times\frac{(2(q-k))!!}{(2(q-k))!}\\
&=(2q)!\sum_{k=0}^q(-1)^k\frac{1}{2^kk!\times 2^{q-k}(q-k)!}\\
&=\frac{(2q)!}{2^qq!}\sum_{k=0}^q(-1)^k\binom{q}{k}\\
&=0\, .
\end{align*}
\end{proof}
Using Claim \ref{cl.determinant_computation}, \eqref{e.betti.pf.1} yields
\[
\E\left[\chi(Z_L)\right]=o\left(\frac{L}{\sqrt{\ln\left(L^{1/2}\right)}}\right)
\]
as announced.
\end{proof}

Finally, we check Theorem \ref{t.main}.

\begin{proof}[Proof of Theorem \ref{t.main}]
The upper bound follows with $C=A_n^0$ by Theorem \ref{t.betti}. For the lower bound, first, by the second point of Theorem \ref{t.covariance} together with the compactness of $\man$, there exist a constant $a=a(\man)<+\infty$ such that the first assumption of Theorem \ref{t.tent} is satisfied by $(f_\lambda)_{\lambda\geq 1}=\left(f_{\sqrt{L}}\right)_{L\geq 1}$ in local charts of some atlas. Let us check that the second assumption of Theorem \ref{t.tent} is also satisfied by this family of fields with $\alpha=2$. Consider $U$ a local chart given by Theorem \ref{t.covariance}. Let $V\subset U$ a convex neighborhood of $0$. Fix $x,y\in U$. For each $t\in [0,1]$ we have $x+t(y-x)\in U$. For such $t$, let $u(t)=\E\left[\left(f_{\sqrt{L}}(x)-f_{\sqrt{L}}(x+t(y-x))\right)^2\right]$. Then, $u$ is twice continuously differentiable and we have $u(0)=u'(0)=0$. Moreover, for each $t\in [0,1]$,
\begin{align*}
u''(t)=&2\E\left[\left(d_{x+t(y-x)}f_{\sqrt{L}}(y-x)\right)^2\right]\\
&+2\E\left[\left(f_{\sqrt{L}}(x+t(y-x))-f_{\sqrt{L}}(x)\right)d^2_{x+t(y-x)}f_{\sqrt{L}}(y-x,y-x)\right]\\
\leq& 2\partial_z\partial_wK_L(z,w)\Big|_{z=w=x+t(y-x)}(y-x,y-x)\\
&+2\left(u(t)\partial_z^2\partial_w^2K_L(z,w)\Big|_{z=w=x+t(y-x)}(y-x,y-x,y-x,y-x)\right)^{1/2}\, .
\end{align*}
Here we used the definition of $K_L$ as well as the Cauchy-Schwarz inequality. Applying the third point of Theorem \ref{t.covariance} to the derivatives of $K_L$ in the right hand side of the last line of the above computation, we have the following estimate: There exists $C=C(n)<+\infty$ such that for each $L\geq 1$, and for any choice of $x,y\in V$,
\[
u''(t)\leq CL|x-y|^2\left(1+\sqrt{u(t)}\right)\, .
\]
Now, applying Taylor's inequality up to order $2$ to $u$ from $0$ to any $t\in [0,1]$, we get
\[
u(t)\leq \frac{C}{2}L|x-y|^2\left(1+\sqrt{\sup_{0\leq s\leq t}u(s)}\right)\, .
\]
Assume now that $\sqrt{L}|x-y|\leq 1$. For each $t\in [0,1]$, let $v(t)=\sup_{0\leq s\leq t}u(s)$. We have shown that for each $t\in [0,1]$,
\[
\frac{v(t)}{1+\sqrt{v(t)}}\leq \frac{C}{2}L|x-y|^2\leq \frac{C}{2}\, .
\]
Comparing the first and third term, we get that $v(t)$ is uniformly bounded. Then, with this information, the first inequality shows that there exists $b<+\infty$ depending only on $C$ such that $v(t)\leq bL|x-y|^2$. In particular, for each $L\geq 1$ and each $x,y\in V$ such that $\sqrt{L}|x-y|\leq 1$,
\[
\E\left[\left(f_{\sqrt{L}}(x)-f_{\sqrt{L}}(y)\right)^2\right]=u(1)\leq v(1)\leq b^2L|x-y|^2\, .
\]
Hence, Theorem \ref{t.tent} does apply. Let $\rho>0$, $\kappa>0$ and $\lambda_0<+\infty$ be as in Theorem \ref{t.tent}. Then, there exists $c'>0$ independent of $\man$  and $\lambda_1=\lambda_1(\man)$ such that for each $L\geq \lambda_0^2\vee\lambda_1^2$, we can find $\lfloor c' L^{n/2} \vol_g(\man)\rfloor$ disjoint balls of radius $\rho L^{-1/2}$ in $\man$. By Theorem \ref{t.tent}, for such $L$, each ball contains a nodal component with probability at least $\kappa\left(\ln\left(L^{1/2}\right)\right)^{-1/2}$ and
\[
\E[N_L]\geq \kappa \lfloor c' L^{n/2} \vol_g(\man)\rfloor\left(\ln\left(L^{1/2}\right)\right)^{-1/2}\, .
\]
This proves the lower bound with for instance $c=c'\kappa/2$. Finally, since $\man$ is compact, the Euclidean diameter of balls in local charts such as the one used above and the Riemmanian diameter are comparable. Since the connected components constructed using Theorem \ref{t.tent} have Euclidean diamter less than $\rho/\lambda$ then we also have the second part of Theorem \ref{t.main}.
\end{proof}

\section{Differential geometry of smooth Gaussian fields}\label{s.geometry}

The object of this section is to apply classical results from the theory of smooth Gaussian fields to our setting. These results will be used in Section \ref{s.upper_bound}. Throughout this section $\man$ will be a smooth manifold of dimension $n>0$ and $f$ be a real-valued Gaussian field on $\man$ with covariance $K$. We will assume throughout that $f$ is almost surely of class $C^2$.\\

\begin{condition}\label{cond.1}
For each $x\in \man$, the Gaussian vector $(f(x),d_xf)$ is non-degenerate.
\end{condition}

\begin{condition}\label{cond.2}
Let $p\in C^\infty(\man)$. For each $x\in \man$ such that $d_xp\neq 0$ and for any connection $\nabla^p$ on $p^{-1}(0)$ defined near $x$, the Gaussian vector $(f(x),d_xf,(\nabla^p df|_{p^{-1}(0)})_x)$ is non-degenerate.
\end{condition}

\begin{remark}
Condition \ref{cond.2} is satisfied for any connection $\nabla^p$ as soon as it is satisfied for one particular connection.
\end{remark}

\begin{remark}
Conditions \ref{cond.1} and \ref{cond.2} are diffeomorphism invariant.
\end{remark}

\begin{remark}
As we shall see in Lemma \ref{l.covariance_convergence} below, for any $p\in C^\infty(\man)$ the field $(f^s_L)$ satisfies Condition \ref{cond.2} for $s=n/2$ and large enough values of $L$.
\end{remark}

\subsection{Manifold versions of classical lemmas for smooth Gaussian fields}\label{ss.lemmas}

In this subsection we state two general lemmas for smooth Gaussian fields. The statements are slightly altered to fit the case of section-valued fields. We simply check that the Euclidean case adapts well to this setting. First, let us introduce the following notation.

\begin{definition}
Let $V,W$ be two $n$-dimensional vector spaces each equipped with a non-zero $n$-form $\omega_V$ and $\omega_W$. Let $A:V\rightarrow W$ be a linear map. We define the determinant $\det_{\omega_V,\omega_W}(A)$ of $A$ from $(V,\omega_V)$ to $(W,\omega_W)$ by the following equation
\[
A^*\omega_W= \det_{\omega_V,\omega_W}(A)\omega_V\, .
\]
\end{definition}

Note that changing the sign of $\omega_V$ or $\omega_W$ only affects the sign of the determinant so we can speak of $|\det_{|\omega_V|,|\omega_W|}(A)|$ even when $|\omega_V|$ and $|\omega_W|$ are given only up to a sign.\\

The following lemma is an adaptation of the classical Kac-Rice formula to crossings of sections of vector bundles. It is the main tool used in the proof of Theorem \ref{t.morse_rice}. Before we state the lemma, we introduce the following terminology.

\begin{definition}
Let $\man$ be a smooth $n$-dimensional manifold and let $E\rightarrow\man$ be a vector bundle on $\man$. For each $x\in \man$, we denote by $E_x$ the fiber of $E$ at $x$. A random $E$-valued field on $\man$ will be a collection of random variables $(F_x)_{x\in\man}$ defined on the same probability space such that for each $x\in\man$, the random variable $F_x$ takes values in $E_x$. A Gaussian $E$-valued field on $\man$ is a random $E$-valued field $(F_x)_{x\in\man}$ on $\man$ such that for each $x_1,\dots,x_k\in\man$, the random variable $(F_{x_1},\dots,F_{x_k})$ is a Gaussian vector in $E_{x_1}\times \dots\times E_{x_k}$.
\end{definition}

We have the following result, which is a variant of Theorem 6.4 of \cite{azws}.

\begin{lemma}\label{l.vector_rice}
Let $\man$ be a smooth $n$-dimensional manifold equipped with a smooth positive density $d\mu$ and let $E\rightarrow \man$ be a vector bundle on $\man$ of rank $n$. We equip $E$ with positive smooth density $d\nu$. Let $F$ be a Gaussian $E$-valued field on $\man$ that is almost surely of class $C^1$ and such that for each $x\in \man$, $F_x$ is a non-degenerate $E_x$-valued Gaussian vector. For each $x\in \man$, let $\gamma_{\nu,F_x}$ be the density of $F_x$ with respect to $d\nu$. Then, for any connection $\nabla^E$ on $E$, any Borel subset $B\subset \man$, any section $\sigma$ of $E$ defined in a neighborhood of $B$ and any $\varphi\in C^\infty_c(T^*\man\otimes E)$,
\begin{multline*}
\E\left[\sum_{x\in B\, |\, F_x=\sigma_x}\varphi\left(x,(\nabla^EF)_x\right)\right]=\\
\int_B\E\left[\varphi\left(x,(\nabla^EF)_x\right)\left|\det_{\mu_x,\nu_x}\left(\left(\nabla^EF\right)_x-\left(\nabla^E\sigma\right)_x\right)\right|\, \Big|\, F_x=\sigma_x\right]\gamma_{\nu,F_x}(\sigma_x)d\mu(x)\, .
\end{multline*}
Here, both quantities may be infinite.
\end{lemma}

Note that $\left(\nabla^EF\right)_x-\left(\nabla^E\sigma\right)_x$ does not depend on $\nabla^E$ at the points $x$ where $F_x=\sigma_x$.

\begin{proof}[Proof of Lemma \ref{l.vector_rice}]
Since both sides of the equality are additive in $B$, we may restrict ourselves to the case where $E\rightarrow \man$ is the trivial bundle $U\times\R^n\rightarrow U$ over some open subset $U\subset\R^n$. Moreover, by considering $\tilde{F}=F-\sigma$, and $\tilde{\varphi}(x,\zeta)=\varphi(x,\zeta-(d_x\sigma))$ it is enough to treat the case where $\sigma = 0$. In this case, $F$ is just an $\R^n$ valued Gaussian field on $U$, and the Rice formula (see Theorem 6.4 of \cite{azws}) applies. Therefore,

\begin{equation}\label{e.vector_rice.1}
\E\left[\sum_{x\in B\, |\, F_x=0}\varphi(x,d_xF)\right]=\int_B\E\left[\varphi(x,d_xF)\left|\det\left(d_xF\right)\right|\, \big|\, F_x=0\right]\gamma_{F_x}(0)dx\, .
\end{equation}

Here, first, $\gamma_{F_x}$ is the density of the measure of $F(x)$ with respect to the Lebesgue measure $dx$. Second, we endowed $T^*_xU$ with the Lebesgue density and $\det$ is the usual determinant, i.e., $|\det|=|\det_{dx,dv}|$. Let $g,h\in C^\infty(U)$ be such that

\[
d\mu(x)=g(x)dx;\, d\nu_x(v)=h(x)dv\, .
\]

Then, for any $L\in \Hom(T_xU,\R^n)$, $|\det(L)|g(x) = |\det_{\nu_x,\mu_x}(L)|h(x)$ and $\gamma_{F_x}=\gamma_{\nu, F_x}h(x)$. Applying these identities to the right hand side of equation \eqref{e.vector_rice.1}, we get

\begin{equation}\label{e.vector_rice.2}
\E\left[\sum_{x\in B\, |\, F_x=0}\varphi(x,d_xF)\right]=\int_B\E\left[\varphi(x,d_xF)\left|\det_{\nu_x,\mu_x}\left(d_xF\right)\right|\, \big|\, F_x=0\right]\gamma_{\nu,F_x}(0)d\mu(x)\, .
\end{equation}

Finally, let $\nabla$ be a connection on $E$. Then, there is exists smooth family $(A_x)_{x\in U}$ of linear maps $A_x:\R^n\rightarrow \Hom(T_xU,\R^n)$ such that for any function $f:U\rightarrow \R^n$ and any $x\in U$,
\[
\nabla f_x = d_xf+(A_x\circ f)(x)\, .
\]
In particular, 
\[
\E\left[\varphi(x,d_xF)\left|\det_{\nu_x,\mu_x}\left(d_xF\right)\right|\, \big|\, F_x=0\right]
=\E\left[\varphi(x,(\nabla F)_x)\left|\det_{\mu_x,\nu_x}\left(\left(\nabla F\right)_x\right)\right|\, \big|\, F_x=0\right]
\]
and in view of equation \eqref{e.vector_rice.2} we are done.
\end{proof}

In Subsection \ref{ss.almost_sure} we will use the following lemma to prove basic facts about the regularity of $Z_f=\{x\in\man\ :\ f(x)=0\}$ and functions defined on it. It is an easy application of Lemma 11.2.10 \cite{adta_rfg}.

\begin{lemma}\label{l.transversality}
Let $\man$ be a smooth manifold of dimension $n$ equipped with a rank $n+1$ vector bundle $E\rightarrow \man$. Let $\nu$ be an arbitrary positive smooth density on $E$. Let $F$ be a random field on $\man$ with values in $E$. Assume that $F$ is almost surely $C^1$ on $\man$ and that the random vector $F(x)$ has a density with respect to $\nu_x$ which is uniformly bounded for $x$ in compact subsets of $\man$. Then, almost surely, for all $x\in \man$, $F(x)\neq 0$.
\end{lemma}
\begin{proof}
Firstly, since $\man$ is paracompact we may assume that $E\rightarrow \man$ is the trivial bundle over some open subset $U\subset\R^n$ so that $F$ is an $\R^{n+1}$ valued Gaussian field on $U$. Next, the assumption on the density of $F_x$ does not depend on the choice of $\nu$ so we may assume that it is the Lebesgue measure. Since $U$ is covered by a countable union of closed balls, it suffices to check the conclusion of the lemma for $T$ a compact subset of $U$ of Hausdorff dimension $n$. The fact that $F$ is almost surely $C^1$ implies that its covariance is continuous. Since for each $x\in T$, $F(x)=(F_1(x),\cdots,F_n(x))$ has locally bounded density and $T$ is compact, its density is bounded on $T$. Finally, since $T$ is compact and the partial derivatives of $F$ are almost surely continuous, they are almost surely bounded on $T$. Hence we can apply Lemma 11.2.10 of \cite{adta_rfg} with $u=0\in\R^{n+1}$.
\end{proof}

\subsection{Almost-sure properties of $Z_f$}\label{ss.almost_sure}

In this subsection, we use Lemma \ref{l.transversality} to prove that $Z_f=\{x\in\man\ :\ f(x)=0\}$ is almost surely smooth and that if we restrict an adequate deterministic function to $Z_f$ then it is almost surely Morse. We begin by treating smoothness.

\begin{lemma}\label{l.as_smooth}
Assume that $f$ satisfies Condition \ref{cond.1}. Then, $Z_f$ is almost surely a $C^2$ hypersurface of $\man$.
\end{lemma}
\begin{proof}
The random field $F=(f,df)$ is a Gaussian random field on $\man$ with values in $\R\times T^*\man$. Moreover, it is almost surely $C^1$ since $f$ is almost surely $C^2$. In addition, since $f$ satisfies Condition \ref{cond.1}, the components of $F$ have positive variances. Therefore, by Lemma \ref{l.transversality}, almost surely, there is no $x\in \man$ such that $F(x)=0$. Therefore, $f$ is almost surely regular on $Z_f$. Since $f$ is of class $C^2$, $Z_f$ is almost surely a smooth hypersurface of $\man$ of class $C^2$.
\end{proof}

We now briefly recall the definition of the Hessian of a function at its critical point.

\begin{definition}
Let $\man$ be a smooth manifold. We equip $T^*\man$ with a connection $\nabla$. Let $f\in C^2(\man)$ and $x\in \man$ be such that $d_xf=0$. Then, $\left(\nabla df\right):T^*_x\man\times T^*_x\man\rightarrow\R$ defines a symmetric bilinear form that does not depend on the choice of connection $\nabla$. We call this map the Hessian of $f$ at $x$ and denote it by $\Hess(f)(x)$. In particular, for any local chart $\psi:U\subset \man\rightarrow\psi(U)\subset\R^n$, any $x\in U$ such that $d_xf=0$ and any $v,w\in T_x\man$,

\begin{equation}\label{e.hess.3}
d^2_{\psi(x)}(f\circ\psi^{-1})(v,w)=\Hess(f)(x)((d_x\psi)^{-1}v,(d_x\psi)^{-1}w))\, .
\end{equation}
\end{definition}

The following lemma will be useful first to prove the next almost sure result about $Z_f$ and to characterize the signature of the Hessian of a function restricted to $Z_f$ in terms of $f$ near this point.

\begin{lemma}[Lemma A.3 \cite{gawe14}]\label{l.hess}
Let $\man$ be an $n$-dimensional smooth manifold. Let $f,p\in C^2(\man)$ and fix $x_0\in \man$. Assume that $d_{x_0}f,d_{x_0}p\neq 0$ and that $d_{x_0}p=\lambda d_{x_0}f$ for some $\lambda\in\R\setminus\{0\}$. Then $L_f=f^{-1}(f(x_0))$ and $L_p=p^{-1}(p(x_0))$ are both smooth in a neighborhood of $x_0$ and
\[
\Hess(p|_{L_f})(x_0)=-\lambda\Hess(f|_{L_p})(x_0)\, .
\]
\end{lemma}

Here, note that the condition $d_{x_0}p=\lambda d_{x_0}f$ implies that both $f|_{L_p}$ and $p|_{L_f}$ are singular at $x_0$ so their Hessian at $x_0$ are well defined bilinear forms on $T_{x_0}L_p$ and $T_{x_0}L_f$ respectively. But these two tangent spaces are naturally isomorphic to the same subspace of $T_{x_0}\R^n$ so it makes sense to compare the two Hessians. Gayet and Welschinger state and show this lemma in a coordinate free language. Here we provide a proof in local coordinates.

\begin{proof}[Proof of Lemma \ref{l.hess}]
Without loss of generality, we may replace $(\man,x_0)$ by $(U,0)$ where $U$ is an open neighborhood of $0$ in $\R^n$. We may also assume $d_0f=d_0p=dx_n$, $f(0)=p(0)=0$. Since $d_0f=dx_n$, the following map is a local diffeomorphism at $0$:
\[
x=(x_1,\dots,x_n)\mapsto(x_1,\cdots,x_{n-1},f(x))\, .
\]
Moreover, its inverse $F$ satisfies $d_0F=Id_n$. The map $F$ is a local diffeomorphism at $0$, say from $0\in W\subset\R^n$ to $x_0\in V\subset U$. Let $h\in C^2(U)$ be such that $d_0h=dx_n$. Then,

\begin{align*}
d^2_0(h\circ F)&=d^2_0h\circ\left(d_0F\right)^{\otimes 2}+d_0h\circ\left(d^2_0F\right)\\
&= d^2_0h+d^2_0F_n\, .
\end{align*}

Take first $h=f$. Then, the left hand side vanishes since $f\circ F(x)=x_n$ and

\[
0 = d^2_0f+d^2_0F_n\, .
\]

Next, take $h=p$. Then,

\[
d^2_0(p\circ F) = d^2_0p+d^2_0F_n\, .
\]

Hence,

\begin{equation}\label{e.hess.1}
d^2_0(p\circ F)=d^2_0p-d^2_0f\, .
\end{equation}

By symmetry of the initial assumptions, if $P$ is the local inverse at $x_0$ of the map

\[
x=(x_1,\dots,x_n)\mapsto(x_1,\cdots,x_{n-1},p(x))\, .
\]

Then,

\begin{equation}\label{e.hess.2}
d^2_0(f\circ P)=d^2_0f-d^2_0p\, .
\end{equation}

Therefore, if $H=T_0L_f=T_0L_p$, we have

\[
\Hess(p|_{L_f})(0)=
d^2_0(p\circ F)|_H=-d^2_0(f\circ P)|_H=-\Hess(f|_{L_p})(0)
\]

where the middle equality follows from \eqref{e.hess.1} and \eqref{e.hess.2} while the two others follow from equation \eqref{e.hess.3}.
\end{proof}

Finally, we prove the following result.

\begin{lemma}\label{l.as_morse}
Assume that $f$ satisfies Condition \ref{cond.2}. Let $p\in C^\infty(\man)$ with an at most countable set of critical points. Let $\man'\subset\man$ be the regular set of $p$. Then, almost surely, $Z_f\subset\man'$ and $p|_{Z_f}$ is a Morse function.
\end{lemma}

\begin{proof}
Firstly, for any critical point $x$ of $p$, by Condition \ref{cond.2}, $f(x)$ is non-degenerate and therefore almost surely non-zero. Since $p$ has at most countably many critical points, almost surely, $Z_f$ stays in $\man'$. By Lemma \ref{l.hess}, for $p|_{Z_f\cap\man'}$, not to be Morse, there must be $(x,v)\in T\man$ such that $f(x)=0$, $d_xf$ vanishes on the kernel of $d_xp$, $d_xp(v)=0$ and $\Hess(f|_{p^{-1}(p(x))})(x)(v,\cdot)=0$. Let us prove that this is almost surely never the case. On the manifold $\man'$, the kernel of $dp$ defines a smooth rank $n-1$ vector bundle $\calK$ on $\man$. Let $S(\calK)$ be the unit sphere bundle of $\calK$ for some auxiliary metric on $\calK$. Let $\nabla^p$ be an auxiliary connection on $T^*p^{-1}(p(x))$. For each $x\in\man'$ and $v\in S_x(\calK)$, let $F(x,v)=\left(f(x),d_xf|_{\calK_x},\left(\nabla^p(df|_{\calK})\right)(x)(v,\cdot)\right)$. Let $\pi$ denote the projection map $S(\calK)\rightarrow\man'$. Then, $(F(x,v))_{(x,v)\in S(\calK)}$ defines a $\pi^*\left(\R\otimes\calK^*\otimes\calK^*\right)$ valued Gaussian field on $S(\calK)$. But $S(\calK)$ is a smooth manifold of dimension $2n-2$ while the image vector bundle has dimension $2n-1$. By Condition \ref{cond.2}, Lemma \ref{l.transversality} applies so, with probability one, the field $F$ does not vanish and $p|_{Z_f\cap\man'}$ is a Morse function.
\end{proof}

\section{The upper bound in the critical case}\label{s.upper_bound}

In this section we apply the results of Section \ref{s.geometry} to prove Theorem \ref{t.morse_rice}. First, in Subsection \ref{ss.covariance} we use Theorem \ref{t.covariance} to compute the asymptotic covariance of the two-jet of $f_L$ at a given point. Then, in Subsection \ref{ss.rice_proof} we prove Theorem \ref{t.morse_rice} using this computation on the integral formula provided by Lemma \ref{l.vector_rice}. In this subsection, we consider a closed Riemmanian manifold $(\man,g)$ of positive dimension $n$ and consider the family of fields $(f^s_L)_{L\geq 1}$ defined in Section \ref{s.introduction} for $s=n/2$. Since we have fixed $s$, we write $f_L$ instead of $f^s_L$ and $K_L$ instead of $K^s_L$ for the rest of the subsection. Moreover, set $Z_L=Z_{f_L}=\{x\in\man\ :\ f_L(x)=0\}$. Recall that $|dV_g|$ is the Riemmanian density induced by $g$ on $\man$.

\subsection{Covariance computations}\label{ss.covariance}

The object of this subsection is to prove the following Lemma.
\begin{lemma}\label{l.covariance_convergence}
Let $p\in C^\infty(\man)$ be a Morse function on $\man$. Let $\calK$ be the vector bundle on $\man'=\{x\in\man,\, d_xp\neq 0\}$ whose fiber above $x$ is $Ker(d_xp)\subset T_x\man$. Fix $x_0\in\man$. There exist $\nabla^p$ a connection on $\calK$ and local coordinates $x=(x_1,\dots,x_n)$ defined on $0\in U\subset\R^n$ centered at $x_0$ such that the density $|dV_g|$ agrees with the Lebesgue measure in these coordinates and the following holds. Let $U'$ be the regular set of $p$ in these coordinates. For each $x\in U'$ and $L\geq 1$, define the centered Gaussian vector $(X^L(x),Y^L(x),Z^L(x))$ with values in $\R\times\R^n\times\textup{Sym}_{n-1}(\R)$ as follows. Let $X^L(x)=\frac{1}{\sqrt{\ln\left(L^{1/2}\right)}}f_L(x)$. Next, let $Y^L(x)=(Y^L_1(x),\cdots,Y^L_n(x))$ be $L^{-1/2}d_xf_L$ seen as a $n$-uple in the local coordinates. Finally, let $Z^L(x)$ be $L^{-1}\nabla^p(df_L|_\calK)(x)|_{\calK_x}$ seen as a symmetric $(n-1)$-matrix in the local coordinates Then, for any $i,j,k,l\in\{1,\dots,n-1\}$ such that $i\neq j$ and $(i,j)\neq (k,l)$, uniformly for $x\in U'$, the covariance matrix of $(X^L(x),Y^L(x),Z^L(x))$ converges as $L\rightarrow +\infty$ to the following matrix.
\[
c_n\times \left(
\begin{array}{ccc}
n & 0 & 0\\
0    & \frac{1}{2} Id_n & 0\\
0    & 0  & \frac{1}{4(n+2)}\Xi
\end{array}
\right)
\]
where $c_n$ is the (positive) constant defined in Lemma \ref{l.integrals}. Moreover, $\Xi$ is the symmetric matrix indexed by the pairs $((i,j),(k,l))\in\{1,\dots,n-1\}^2\times\{1,\dots,n-1\}^2$ such that $i\leq j$ and $k\leq l$ defined as $\frac{1}{4(n+2)}\Xi_{ij,kl}=\E[Z^L(x)_{ij}Z^L(x)_{kl}]$. We have for any $i,j,k,l\in\{1,\dots,n-1\}$ such that $i<j$, $k\leq l$ and $(i,j)\neq (k,l)$, $\Xi_{ii,jj}=\Xi_{ij,ij}=1$, $\Xi_{ii,ii}=3$ and $\Xi_{ij,kl}=0$. In particular, for $L$ large enough, the field $f_L$ satisfies Condition \ref{cond.2} on $U'$.
\end{lemma}
\begin{proof}[Proof of Lemma \ref{l.covariance_convergence}]
We start with $\nabla_1^p$ a connection on $\calK$. Fix $x_0\in\man$ and consider a local coordinate patch $\tilde{U}$ at $x_0$ given by Theorem \ref{t.covariance} that is also isometric at $x_0$ (i.e., we assume that the differential at $x_0$ of the diffeomorphism defining the local coordinates is an isometry from $(T_{x_0}\man,g_{x_0})$ to $\R^n$ equipped with the Euclidean scalar product).  We have, in this set of local coordinates, $\nabla^p_1(df_L|_{\calK})(x)|_{\calK_x}=d^2_xf_L|_{\calK_x}+(A_x(d_xf_L|_{\calK_x}))|_{\calK_x}$ where $A\in\Gamma(U';T^*U\otimes\calK^*\otimes\calK)$. Let $\chi\in C^\infty_c(\tilde{U})$ be equal to one in a neighborhood $U$ of $0$ and let $\nabla^p=\nabla^p_1-\chi A$. Then, $\nabla^p$ defines a connection on $\calK$. Let $i,j,k,l\in\{1,\dots,n-1\}$ such that $i\neq j$ and $(i,j)\neq (k,l)$. Moreover, with this choice of $\nabla^p$, we have, for each $x\in U'$, $Z_L(x)=L^{-1}d^2_xf_L|_{K_x}$. Therefore, the joint covariance of $(X^L(x),Y^L(x),Z^L(x))$ is equal to
\[
\left(
\begin{array}{ccc}
\frac{1}{\ln\left(L^{1/2}\right)}K_L(x,x) & \frac{1}{L^{1/2}\sqrt{\ln\left(L^{1/2}\right)}}d_yK_L(x,x) & \frac{1}{L\sqrt{\ln\left(L^{1/2}\right)}}d^2_yK_L(x,x)\\
\frac{1}{L^{1/2}\sqrt{\ln\left(L^{1/2}\right)}}d_yK_L(x,x) & \frac{1}{L}d_xd_yK_L(x,x) & \frac{1}{L^{3/2}}d_xd^2_yK_L(x,x)\\
\frac{1}{L\sqrt{\ln\left(L^{1/2}\right)}}d^2_xK_L(x,x) & \frac{1}{L^{3/2}}d^2_xd_yK_L(x,x) & \frac{1}{L^2}d^2_xd^2_yK_L(x,x)
\end{array}
\right)
\]
for any $x\in U$. By applying the estimates of the second and third of Theorem \ref{t.covariance} and since, by parity, the integrals $\int_{|\xi|^2\leq 1}|\xi|^{-n}(i\xi)^\alpha(-i\xi)^\beta d\xi$ vanish whenever $\alpha+\beta\in\N^n$ has at least one odd component, we get, for any $i,j,k,l\in\{1,\dots,n\}$ such that $i<j$, $k\leq l$ and $(i,j)\neq (k,l)$, uniformly for $x\in U'$:
\begin{align*} 
\lim_{L\rightarrow+\infty}\E\left[X^L(x)^2\right]&=\frac{|S^{n-1}|}{(2\pi)^n}\\
\lim_{L\rightarrow+\infty}\E\left[X^L(x)Y^L_i(x)\right]&=0\\
\lim_{L\rightarrow+\infty}\E\left[Y^L_i(x)Z^L_{kl}(x)\right]&=0\\
\lim_{L\rightarrow+\infty}\E\left[X^L(x)Z^L_{kl}(x)\right]&=0\\
\lim_{L\rightarrow+\infty}\E\left[Y_i^L(x)^2\right]&=\frac{1}{(2\pi)^n}\int_{|\xi|^2\leq 1}\frac{\xi_i^2}{|\xi|^n}d\xi\\
\lim_{L\rightarrow+\infty}\E\left[Y_i^L(x)Y_j^L(x)\right]&=0\\
\text{if }k,l\leq n-1,\, \lim_{L\rightarrow+\infty}\E\left[Z_{kk}^L(x)Z_{ll}^L(x)\right]=\lim_{L\rightarrow+\infty}\E\left[Z_{kl}^L(x)^2\right]&=\frac{1}{(2\pi)^n}\int_{|\xi|^2\leq 1}\frac{\xi_k^2\xi_l^2}{|\xi|^n}d\xi\\
\text{ if $i,j,k,l\leq n-1$, }\lim_{L\rightarrow+\infty}\E\left[Z_{ij}^L(x)Z_{kl}^L(x)\right]&=0\, .
\end{align*}
The first statement then follows by the computations carried out in Lemma \ref{l.integrals}. To check Condition \ref{cond.2} we check that the limit law is non-degenerate. Aside from the diagonal coefficients of the symmetric matrix component, all the components are independent with positive variance. As for the diagonal components of the symmetric matrix, their covariance is a positive multiple of $2Id_{n-1}+J_{n-1}$ where $J_{n-1}$ is the $(n-1)\times(n-1)$ matrix whose coefficients are all equal to $1$. But $J_{n-1}$ is the covariance of the constant Gaussian vector with unit variance so it is non-negative. Therefore, $2Id_{n-1}+J_{n-1}$ is positive definite.
\end{proof}

The following lemma contains the integral calculations needed for the proof of Lemma \ref{l.covariance_convergence} above.
\begin{lemma}\label{l.integrals}
Fix $n\geq 2$. We define the universal constant $c_n$ as follows. 
\[
c_n= \frac{2^{(n+1)/2}((n+1)/2)!}{(n+1)!}\sqrt{\frac{2}{\pi}}\text{ if $n$ is odd and } \frac{1}{2^{n/2}(n/2)!}\text{ if $n$ is even. }
\]
Let $B^n$ denote the Euclidean unit ball in $\R^n$. Then, for any $i,j\in\{1,\dots,n\}$ distinct,
\begin{align*}
\frac{|S^{n-1}|}{(2\pi)^n}&= n c_n\\
\frac{1}{(2\pi)^n}\int_{B^n}\frac{\xi_i^2}{|\xi|^n}d\xi &=\frac{1}{2}c_n\\
\frac{1}{(2\pi)^n}\int_{B^n}\frac{\xi_i^2\xi_j^2}{|\xi|^n}d\xi &=\frac{1}{4(n+2)}c_n\\
\frac{1}{(2\pi)^n}\int_{B^n}\frac{\xi_i^4}{|\xi|^n}d\xi &=\frac{3}{4(n+2)}c_n\,.
\end{align*}
\end{lemma}
\begin{proof}[Proof of Lemma \ref{l.integrals}]
First, by a polar change of coordinates we get
\begin{align*}
\int_{B^n}\frac{\xi_i^2}{|\xi|^n}d\xi &=\frac{1}{2}\int_{S^{n-1}}\omega_i^2d\omega\\
\int_{B^n}\frac{\xi_i^2\xi_j^2}{|\xi|^n}d\xi &=\frac{1}{4}\int_{S^{n-1}}\omega_i^2\omega_j^2d\omega\\
\int_{B^n}\frac{\xi_i^4}{|\xi|^n}d\xi &=\frac{1}{4}\int_{S^{n-1}}\omega_i^4d\omega\,.
\end{align*}
Here $d\omega$ is the surface area for the unit sphere in $\R^n$. Moreover $|S^{n-1}|=\int_{S^{n-1}}d\omega$. To compute the integrals over the sphere, we compare them to moments of Gaussian random variables. Let $X$ be a centered Gaussian vector in $\R^n$ with covariance $Id_n$. Another polar change of coordinates yields, for $1\leq i<j\leq n$,
\begin{align*}
1=\E[1]&=\frac{1}{(2\pi)^{n/2}}\int_0^{+\infty}t^{n-1}e^{-t^2/2}dt\int_{S^{n-1}}d\omega\\
1=\E[X_i^2]&=\frac{1}{(2\pi)^{n/2}}\int_0^{+\infty}t^{n+1}e^{-t^2/2}dt\int_{S^{n-1}}\omega_i^2d\omega\\
1=\E[X_i^2X_j^2]&=\frac{1}{(2\pi)^{n/2}}\int_0^{+\infty}t^{n+3}e^{-t^2/2}dt\int_{S^{n-1}}\omega_i^2\omega_j^2d\omega\\
3=\E[X_i^4]&=\frac{1}{(2\pi)^{n/2}}\int_0^{+\infty}t^{n+3}e^{-t^2/2}dt\int_{S^{n-1}}\omega_i^4d\omega\, .
\end{align*}
Now, for each $k\in\N$, let $J_k=\int_0^{+\infty}t^ke^{-t^2/2}dt$. By integration by parts, for each $k\in\N$, we have $J_{k+2}=(k+1)J_k$. From this we deduce the following:
\[
J_k=\frac{1}{2^{k/2}}\frac{k!}{(k/2)!}\sqrt{\frac{\pi}{2}}\text{ if $k$ is even and }2^{(k-1)/2}((k-1)/2)!\text{ if $k$ is odd.}
\]
With this notation,
\begin{align*}
\int_{S^{n-1}}d\omega=\frac{(2\pi)^{n/2}}{J_{n-1}}&=\frac{n(2\pi)^{n/2}}{J_{n+1}}\\
\int_{S^{n-1}}\omega_i^2d\omega&=\frac{(2\pi)^{n/2}}{J_{n+1}}\\
\int_{S^{n-1}}\omega_i^2\omega_j^2d\omega=\frac{(2\pi)^{n/2}}{J_{n+3}}&=\frac{(2\pi)^{n/2}}{(n+2)J_{n+1}}\\
\int_{S^{n-1}}\omega_i^4d\omega=\frac{3(2\pi)^{n/2}}{J_{n+3}}&=\frac{3(2\pi)^{n/2}}{(n+2)J_{n+1}}\, .
\end{align*}
Replacing these expressions in the original integrals yields the desired result (with $c_n=1/J_{n+1}$).
\end{proof}
\subsection{Proof of Theorem \ref{t.morse_rice}}\label{ss.rice_proof}

In this subsection we prove Theorem \ref{t.morse_rice}. The proof relies on Lemmas \ref{l.vector_rice} and \ref{l.covariance_convergence}. Let $p\in C^\infty(\man)$ with an at most countable number of critical points. Let $\man'\subset\man$ be its regular set and let $\calK$ be the sub-bundle of $T\man'$ defined by the kernel of $dp$. For each $i\in\{0,\cdots,n-1\}$, $L\geq 1$ and $B\subset\man$ Borel subset, let $\nu_i(p,f_L,B)$ be the number of critical points of index $i$ of $p|_{Z_f}$ inside $B$. In addition to previous results will need the following elementary lemma:

\begin{lemma}\label{l.expectation_convergence}
Let $(X^L_t)_{t\in T,L\geq 1}=(X_{1,t}^L,X_{2,t}^L)_{t\in T,L\geq 1}$ be a family of centered Gaussian vectors in $\R^n\times\R^m$ with covariances $\Sigma_{t,L}$. Here $T$ is any index set. Assume that, uniformly for $t\in T$, the sequence $(\Sigma_{t,L})_{L\geq 1}$ converges to some covariance matrix $\Sigma$ corresponding to the Gaussian vector $X=(X_1,X_2)$ such that the vector $X_2$ is non-degenerate. Let $f:\R^n\times\R^m\rightarrow\R$ be a measurable function such that for each $\eps>0$ there exist $c=c(\eps)<+\infty$ for which $\forall x\in\R^n\times\R^m$, $|f(x)|\leq ce^{\eps|x|^2}$. Then, uniformly for $t\in T$,
\[
\lim_{L\rightarrow +\infty}\E\left[f(X^L_t)\, |\, X_{2,t}^L=0\right]=\E\left[f(X)\, |\, X_2=0\right]\, .
\]
\end{lemma}
\begin{proof}[Proof of Lemma \ref{l.expectation_convergence}]
Apply the regression formula (see Proposition 1.2 of \cite{azws}) to $(X^L_t)$ and use the dominated convergence theorem on $f$ times the conditional density with respect to the Lebesgue measure on $\R^n$. The fact that $X_{2,t}$ is uniformly non-degenerate guarantees that for large enough values of the vectors $X_{2,t}^L$ are all non-degenerate and that the conditional inverse covariances of $X_t^L$ are uniformly bounded from below by a positive multiple of $Id_n$. This and the sub-exponential bound on $f$ guarantees the uniform integrability needed for dominated convergence.
\end{proof}

We are ready to prove Theorem \ref{t.morse_rice}.

\begin{proof}[Proof of Theorem \ref{t.morse_rice}]
By Lemma \ref{l.covariance_convergence} and compactness of $\man$, for $L$ large enough, $f_L$ satisfies Condition \ref{cond.2} so the smoothness of $Z_L$ and the fact that $p|_{Z_L}$ is almost surely a Morse function follows from Lemmas \ref{l.as_smooth} and \ref{l.as_morse} respectively. By Lemma \ref{l.as_morse}, almost surely $Z_f\subset\man'$ so it is enough to treat the case where $B\subset\man'$. Secondly, the quantities on both sides of equation \eqref{e.morse_rice} are (at least finitely) additive in $B$ so it is enough to prove the result for $B$ inside any local chart of some atlas. Fix $i\in\{0,\cdots,n-1\}$. Let  $x_0\in\man$ and consider $\nabla^p$ a connection on $\calK$ and $x=(x_1,\cdots,x_n)$ the local coordinates centered at $x_0$ provided by Lemma \ref{l.covariance_convergence}, defined on $0\in U\subset\R^n$. Let $U'$ be the regular set of $p$ in these coordinates. Let $|dx|$ be the Lebesgue measure on $U$. Let $|dt|$ be the Lebesgue density on the trivial bundle $\underline{\R}=\R\times U$ on $U$. On $U'$, the Euclidean scalar product restricts to the fibers of $\calK$ (resp. $\calK^\perp$, $\calK^*$) and defines a density $\left|d\tilde{x}\right|$ (resp. $\left|dx^\perp\right|$, $\left|dx^*\right|$). Let $\nu = |dt|\otimes \left|dx^*\right|$. The product $E=\underline{\R}\otimes\calK^*$ is a rank $n$ vector bundle on $U'$. Fix $B\subset U'$. In order to compute $\E\left[m_i(p,f_L,B)\right]$, we wish to apply Lemma \ref{l.vector_rice}. However indicator that the Hessian of $f_L$ restricted to $\calK$ has signature $i$ is discontinuous. We need to approximate it by suitable test functions and justify the convergence of the formula. In other words, we need to prove the following claim:
\begin{claim}\label{cl.checking_kac_rice}
Given $x\in U$, conditionally on the event that $f(x)=0$ and $d_xf|_{\cal K}=0$, let $S^L_i(x)$ be the event that $\Hess(p|_{Z_L})$ has signature $i$. The quantity $\E\left[m_i(p,f_L,B)\right]$ is the integral against $|dx|$ of the following density
\begin{equation}\label{e.crit_count.1}
\E\left[\left|\det_{|dx|,\nu}\left(\left(d_xf_L,\nabla^p df_L|_{\calK_x}\right)(x)\right)\right|\one\left[S^L_i(x)\right]\, \Big|\, f_L(x)=0,\, d_xf_L|_{\calK}=0\right]\gamma_{\nu,(f_L(x),d_xf_L|_{\calK_x})}(0)\, .
\end{equation}
\end{claim}
\begin{proof}
In the bundle $T^*\man\otimes E$, ignoring the $\underline{\R}$ factor and restricting $T^*\man$ to $\calK^*$, we obtain a map $\rho:E\rightarrow\calK^*\otimes\calK^*$.  Let $\varphi\in C^\infty_c(\calK^*\otimes\calK^*)$, which we see as a function in $C^\infty(T^*\man\otimes E)$ composed with the aforementioned map $\rho$. We apply Lemma \ref{l.vector_rice} to the sections $(F_x)_x$ and $(\sigma_x)_x$ of $E$ defined by $F_x:=(f_L(x),d_xf_L|_{\calK_x})$ and $\sigma_x=0$. The lemma applies for $L$ large enough by Lemma \ref{l.covariance_convergence}. We deduce that
\begin{equation}\label{e.claim_morse.1}
\E\left[\sum_{x\in U'\ :\ f_L(x)=0, d_xf_L|_{\calK_x}=0}\varphi\left(\nabla^pdf_L(x)\right)\right]
\end{equation}
is the integral over $B$ against $|dx|$ of the following density
\begin{multline}\label{e.claim_morse.2}
\E\left[\varphi\left(\nabla^pdf_L(x)\right)\left|\det_{|dx|,\nu}\left(\left(d_xf_L,\nabla^p df_L|_{\calK_x}\right)(x)\right)\right|\, \Big|\, f_L(x)=0,\, d_xf_L|_{\calK_x}=0\right]\gamma_{\nu,(f_L(x),d_xf_L|_{\calK_x})}(0)\, .
\end{multline}

Now fix $i\in\{0,\dots,n-1\}$ and let $\varphi^i:\calK^*\otimes\calK^*\rightarrow\{0,1\}$ be defined as $\varphi^i((x,\eta_x))$ is the indicator that the bilinear form $\eta_x$ is non-degenerate and has signature $i$. The function $\varphi^i$ is bounded and continuous on $\calK^*\otimes\calK^*$ except on the set $\mathfrak{D}\subset \calK^*\otimes\calK^*$ of $(x,\eta_x)$ such that said bilinear form is degenerate. Let $(\varphi_l)_{l\in\N}$ be a sequence of functions in $C^\infty_c(\calK^*\otimes\calK^*)$ taking values in $[0,1]$ converging pointwise to $\varphi^i$ on the complement of $\mathfrak{D}$. We study \eqref{e.claim_morse.1} and \eqref{e.claim_morse.2} with $\varphi=\varphi_l$ and claim that as $l\rightarrow+\infty$ we obtain the same statement with $\varphi=\varphi^i$. Let us first check that \eqref{e.claim_morse.1} . Note that applying \eqref{e.claim_morse.1} with $\varphi=1$ we deduce that the random variables
\begin{equation}\label{e.cv_count.1}
\sum_{\{x\in U'\ :\ f_L(x)=0, d_xf_L|_{\calK_x}=0\}}\varphi_l\left(\nabla^pdf_L(x)\right)
\end{equation}
for $l\in\N$ are uniformly bounded by the integrable random variable
\[
\textup{Card}\{x\in U'\ :\ f_L(x)=0, d_xf_L|_{\calK_x}=0\}\, .
\]
Moreover, by Condition \ref{cond.2}, we may apply Lemma \ref{l.transversality} to $F_x=(f_L(x),d_xf_L|_{\calK_x},\det(\nabla^pd_xf_L))$ and deduce that a.s., $\nabla^p d_xf_L\notin\mathfrak{D}$. Thus, the random variable\eqref{e.cv_count.1} converges a.s. to
\[
\sum_{\{x\in U'\ :\ f_L(x)=0, d_xf_L|_{\calK_x}=0\}}\varphi^i\left(\nabla^pdf_L(x)\right)
\]
as $l\rightarrow+\infty$. Since the sequence is uniformly integrable, the expectation \eqref{e.claim_morse.1} with $\varphi=\varphi_l$ converges to the same quantity with $\varphi=\varphi^i$. To show the convergence of the right-hand side, we follow a similar strategy. As before, uniform boundedness follows by replacing $\varphi$ by the constant function equal to $1$ in the expectation of \eqref{e.claim_morse.2}. Next, Condition \ref{cond.2} implies that, Conditionally on $f_L(x)=0$ and $d_xf_L|_{\calK_x}=0$, the Gaussian vector $\nabla^p d_xf$ is non-degenerate. In particular, a.s., $\nabla^p d_xf\notin\mathfrak{D}$ so $\varphi_l(\nabla^p d_xf)$ converges a.s. to $\varphi^i(\nabla^p d_xf)$ as $l\rightarrow+\infty$. We conclude that for each $x\in B$, the quantity \eqref{e.claim_morse.2} with $\varphi=\varphi_l$ converges when $l\rightarrow+\infty$ to the same quantity with $\varphi=\varphi^i$. Moreover, it is uniformly bounded by
\[
\Psi(x):=\E\left[\left|\det_{|dx|,\nu}\left(\left(d_xf_L,\nabla^p df_L|_{\calK_x}\right)(x)\right)\right|\, \Big|\, f_L(x)=0,\, d_xf_L|_{\calK_x}=0\right]\gamma_{\nu,(f_L(x),d_xf_L|_{\calK_x})}(0)
\]
To show convergence of the integrals it is enough to show that $\Psi$ is integrable on $B$. But by the regression formula (see Proposition 1.2 of \cite{azws}), $\Psi$ depends continuously on $x$. Thus, by compactness of $\man$, it is bounded and integrable on $B$. Thus, the quantity \eqref{e.claim_morse.1} with $\varphi=\varphi^i$ is the integral over $B$ of $|dx|$ against the density \eqref{e.claim_morse.2} with $\varphi=\varphi^i$. We conclude the proof of the claim by observing that by definition, $\varphi^i(\nabla^p d_xf)=\one\left[S_i^L(x)\right]$ and $m_i(p,f_L,B)=\sum_{\{x\in B\ : \ f_L(x)=0,\ d_xf_L|_{K_x}=0\}}\varphi^i(\nabla^p d_xf)$.
\end{proof}

Having established Claim \ref{cl.checking_kac_rice}, we set about simplifying the density \eqref{e.crit_count.1}. Note that $|dx|=\left|dx^\perp\right|\otimes\left|d\tilde{x}\right|$ so conditionally on $d_xf_L|_{\calK_x}=0$,

\[
\left|\det_{|dx|,\nu}\left(\left(d_xf_L,\nabla^p df_L|_{\calK_x}\right)(x)\right)\right|=\|d_xf_L\|_{eucl}\left|\det_{\left|d\tilde{x}\right|,\left|dx^*\right|}(\nabla^p(df_L|_{\calK})(x)|_{\calK_x})\right|\, .
\]

For any $x\in U'$ and $L\geq 1$, let $(X^L(x),Y^L(x),Z^L(x))$ be as in Lemma \ref{l.covariance_convergence}. Let $\tilde{Y}^L(x)$ be the coordinates of $Y^L(x)|_{\calK_x}$ in some orthonormal basis of $\calK_x^*$ and let $\Sigma^L(x)$ be the covariance of $(X^L(x),\tilde{Y}^L(x))$. Then,

\[
\|d_xf_L\|_{eucl}\left|\det_{\left|d\tilde{x}\right|,\left|dx^*\right|}(\nabla^p(df_L|_{\calK})(x)|_{\calK_x})\right|=L^{n-1/2}\|Y^L(x)\|_{eucl}\left|\det\left(Z^L(x)\right)\right|
\]

and

\[
\gamma_{\nu,(f_L(x),d_xf|_{\calK_x})}(0)=\frac{1}{(2\pi)^{n/2}L^{(n-1)/2}\sqrt{\ln\left(L^{1/2}\right)}\sqrt{\det(\Sigma^L(x))}}\, .
\]

Therefore, by equation \eqref{e.crit_count.1}, $\E\left[m_i(p,f_L,B)\right]$ is the integral over $B$ and against $|dx|$ of the following density:

\begin{equation}\label{e.crit_count.2}
\frac{\E\left[\|Y^L(x)\|_{eucl}\left|\det\left(Z^L(x)\right)\right|\one\left[S^L_i(x)\right]\, \big|\, X^L(x)=0,\, Y^L(x)|_{\calK_x}=0\right]}{(2\pi)^{n/2}\sqrt{\det(\Sigma^L(x))}}\frac{L^{n/2}}{\sqrt{\ln\left(L^{1/2}\right)}}\, .
\end{equation}

By Lemma \ref{l.covariance_convergence}, uniformly for $x\in U'$,

\[
\lim_{L\rightarrow+\infty}\det\left(\Sigma^L(x)\right)=n2^{1-n}c_n^n\, .
\]

To deal with the expectation, note first that by Lemma \ref{l.hess}, the event $S^L_i(x)$ is exactly the event that, either $Y^L(x)$ is a positive multiple of $d_xp$ and the signature of $Z^L(x)$ is $n-1-i$, or it is a negative multiple and the signature is $i$. Let $(X^\infty,Y^\infty,Z^\infty)$ be the centered Gaussian vector with values in $\R\times\R^n\times\textup{Sym}_{n-1}(\R)$ with the following covariance structure. The three components $X^\infty$, $Y^\infty$ and $Z^\infty$ are independent. $X^\infty$ has variance $nc_n$, $Y^\infty$ has covariance $(c_n/2)I_n$ and the covariance $\frac{c_n}{4(n+2)}\Xi$ of $Z^\infty$ is determined by the following relations: for any $i,j,k,l\in\{1,\dots,n-1\}$ such that $i<j$, $k\leq l$ and $(i,j)\neq (k,l)$,
\begin{align*}
\frac{c_n}{4(n+2)}\Xi_{ii,jj}=\frac{c_n}{4(n+2)}\Xi_{ij;ij}:=\E\left[Z^\infty_{ii}Z^\infty_{jj}\right]&=\E\left[(Z^\infty_{ij})^2\right]=\frac{c_n}{4(n+2)}\\
\frac{c_n}{4(n+2)}\Xi_{ii,ii}:=\E\left[(Z^\infty_{ii})^2\right]&=3\frac{c_n}{4(n+2)}\\
\frac{c_n}{4(n+2)}\Xi_{ij,kl}:=\E\left[Z^\infty_{ij}Z^\infty_{kl}\right]&=0\, .
\end{align*}
 By Lemmas \ref{l.covariance_convergence} and \ref{l.expectation_convergence}, we have, uniformly for $x\in U'$,

\begin{multline*}
\E\left[\|Y^L(x)\|_{eucl}\left|\det\left(Z^L(x)\right)\right|\one\left[S^L_i(x)\right]\, \Big|\, X^L(x)=0,\, Y^L(x)|_{\calK_x}=0\right]\\
\xrightarrow[L\to +\infty]{}\\
\E\left[\|Y^\infty\|_{eucl}\left|\det\left(Z^\infty\right)\right|\one\left[S_i^\infty\right]\, \big|\, X^\infty=0,\, Y^\infty|_{\calK_x}=0\right]
\end{multline*}
where $S^\infty_i$ is the event that either $Z^\infty$ has signature $n-1-i$ and $Y^\infty$ is a positive multiple of $d_xp$ or that its signature is $i$ and $Y^\infty$ is a negative multiple of $d_xp$. Since the components of $(X^\infty,X^\infty,Z^\infty)$ are independent, the above limit equals

\begin{equation}\label{e.crit_count.3}
\E\left[\|Y^\infty\|_{eucl}\, \big|\, Y^\infty|_{\calK_x}=0\right]\E\left[\left|\det\left(Z^\infty\right)\right|\one[\sgn\left(Z^\infty\right)=i]\right]\, .
\end{equation}

Let $M=\sqrt{4(n+2)/c_n}Z^\infty$ so that $M$ has covariance $\Xi$. Then, since $Y^\infty$ has covariance $(c_n/2)Id_n$, the quantity \eqref{e.crit_count.3} equals

\[
(c_n/\pi)^{1/2}(c_n/(4(n+2)))^{(n-1)/2}\E\left[\left|\det\left(M\right)\right|\one[\sgn\left(M\right)=i]\right]\, .
\]

Therefore, by equations \eqref{e.crit_count.2} and \eqref{e.crit_count.3}, as $L\rightarrow +\infty$,
\[
\E\left[m_i(p,f_L,B)\right]\sim C_n\vol_{eucl}(B)\E\left[\left|\det\left(M\right)\right|\one[\sgn\left(M\right)=i]\right]\frac{L^{n/2}}{\sqrt{\ln\left(L^{1/2}\right)}}
\]
where
\[
C_n=\frac{1}{\sqrt{\pi^{n+1}2^{2n-1}n(n+2)^{n-1}}}\, .
\]
To conclude note that in the coordinates given by Lemma \ref{l.covariance_convergence}, the density $|dV_g|$ corresponds to the Lebesgue density so $\vol_{eucl}(B)=\vol_g(B)$.
\end{proof}

\section{The lower bound in the critical case}\label{s.lower_bound}

The object of this section is to prove Theorem \ref{t.tent}. The proofs of this section do not rely on any result from the rest of the article. First, in Subsection \ref{ss.inequalities} we prove two elementary inequalities. Then, we use these to prove Theorem \ref{t.tent} in Subsection \ref{ss.tent}.

\subsection{Two useful Gaussian inequalities}\label{ss.inequalities}

In this subection, we state two inequalities that follow easily from known results. The first is an upper bound for the concentration of the maximum and combines the Fernique inequality with the Borell-TIS inequality.

\begin{lemma}\label{l.tail_of_the_supremum}
Let $g$ be a centered Gaussian field on a bounded subset $V$ of $\R^n$. Assume that there exist $0<\sigma,D<+\infty$ and $\alpha\in]0,2]$ such that for all $x,y\in V$
\begin{align*}
\E\left[g(x)^2\right]&\leq \sigma^2\\
\E\left[(g(x)-g(y))^2\right]&\leq D^2|x-y|^\alpha\, .
\end{align*}
Then, $g$ is almost surely bounded, its supremum $M$ has finite expectation and there exists $C=C(V,\alpha)<+\infty$ such that
\[
\E\left[M\right]\leq CD\, .
\]
Moreover, for each $u>0$,

\[
\prob\left[M\geq CD+u\right]\leq 2e^{-\frac{1}{2\sigma^2}u}\, .
\]

\end{lemma}

\begin{proof}

By Theorem 2.9 of \cite{azws}, it is enough to obtain a uniform bound on the expectation of $M$. Let $(X(x))_{x\in \overline{V}}$ be an $n$-dimensional fractional Brownian motion of index $\alpha$ on $\overline{V}$ (see for instance Definition 3.3.1 \cite{cohen_istas}), that is, $X$ is a centered Gaussian field which is almost surely continuous on $V$ and whose covariance is
\[
\E[X(x)X(y)]=\frac{1}{2}\left[|x|^\alpha+|y|^\alpha-|x-y|^\alpha\right]\, .
\]
Since $X$ is almost surely continuous and $V$ is bounded, its maximum on $\overline{V}$ is almost surely finite. Since it is a Gaussian field, its maximum has finite expectation (see once more Theorem 2.9 of \cite{azws}). Let $m=\E[\max_{\overline{V}} X]$. For any $x,y\in V$,
\[
\E\left[(D X(x)-D X(y))^2\right]=D^2 |x-y|^\alpha\geq \E\left[(g(x)-g(y))^2\right]
\]
so that, by the Sudakov-Fernique inequality (see Theorem 2.4 of \cite{azws}),
\[
\E[M]\leq \E[\sup_V D X]=D m<+\infty
\]
and we are done.
\end{proof}

The second lemma deals with a certain type of event that we now define. For any set $T$, we say that an event $A\subset\R^T$ is \textbf{increasing} if for any $x\in A$
\[
\{y\in\R^T\, |\, \forall t\in T,\, y(t)\geq x(t)\}\subset A\, .
\]
The following result is essentially due to Loren Pitt and says that Gaussian vectors with non-negative covariance satisfy the FKG inequality. Loren Pitt stated it for finite dimensional Gaussian vectors but the general case follows easily (see for instance Theorem A.4 of \cite{rv17qi}).

\begin{lemma}[\cite{pitt82}]\label{l.FKG}
Let $(X_t)_{t\in T}$ be an a.s. continuous Gaussian random field on a separable topological space $T$ with covariance $\Sigma=(\sigma_{ij})_{ij}$. Assume that for each $i,j\in\{1,\cdots,n\}$, $\sigma_{ij}\geq 0$. Then, for any two increasing events $A,B\subset\R^T$ (measureable with respect to the product $\sigma$-algebra),
\[
\prob[X\in A\cap B]\geq \prob[X\in A]\prob[X\in B]\, .
\]
\end{lemma}

\subsection{Proof of Theorem \ref{t.tent}}\label{ss.tent}

In this subsection, we use the inequalities of Subsection \ref{ss.inequalities} to prove Theorem \ref{t.tent}. Throughout the proof, the constants implied by the $O$'s will be universal constants. The proof of Theorem \ref{t.tent} goes roughly as follows:
\begin{itemize}
\item We show first that for $\rho>0$ large enough, at distance $\rho/\lambda$, we may assume that the field is positively correlated, even after conditioning on $f_\lambda(0)$. This will allow us to use the FKG inequality (Lemma \ref{l.FKG}) on points at distance $\rho/\lambda$ from $0$.
\item To detect a nodal domain, we want to estimate the probability that $f_\lambda(0)<0$ and that $f_\lambda$ be positive on a sphere $S_\lambda$ of radius $O(1/\lambda)$ centered at $0$. The difficulty comes from the fact that $f_\lambda(0)$ will typically take very low values (of order $-\sqrt{\ln(\lambda)}$) which we make the second half of the event very unlikely.
\item To deal with this difficulty, we restrict ourselves to the event that $0>f_\lambda(0)\geq -1$. On this event, we show that $f_\lambda$ being positive on a small spherical cap happens with probability at least $1/3$ (see Claim \ref{cl.tent.prf.1}).
\item Since the field is positively correlated on $S_\lambda$, we can "glue" the events that $f_\lambda$ is positive on spherical caps using Lemma \ref{l.FKG}.
\item We conclude by observing that $\prob[0>f_\lambda(0)\geq -1]\asymp(\ln(\lambda))^{-1/2}$.
\end{itemize}

\begin{proof}[Proof of Theorem \ref{t.tent}]
Take $\rho\geq 1$ a parameter to be fixed later and for each $\lambda\geq\rho$, set $S_\lambda$ the sphere centered at $0$ of radius $\rho/\lambda$. For each $\lambda\geq\rho$ and each $x,y\in S_\lambda$,
\begin{align*}
\E\left[f_\lambda(0)^2\right]&=\ln(\lambda) + O(a)\\
\E\left[f_\lambda(0)f_\lambda(x)\right]&=\ln(\lambda)-\ln(\rho) + O(a)\\
\E\left[f_\lambda(x)f_\lambda(y)\right]&\geq \ln(\lambda)-\ln(\rho) +O(a)\, .
\end{align*}
In particular, there exists $\lambda_1(\rho,a)<+\infty$ such that for $\lambda\geq \lambda_1(a,\rho)$, $f_\lambda(0)$ is non-degenerate. By the regression formula (see Proposition 1.2 of \cite{azws}), the field $(f_\lambda(x))_{x\in S_\lambda}$ conditioned on $f_\lambda(0)$ is an almost surely Gaussian field with mean
\begin{equation}\label{e.tent.prf.1}
\forall x\in S_\lambda,\, \E\left[f_\lambda(x)\, \big|\,f_\lambda(0)\right]=\left[1+O\left((\ln(\lambda))^{-1}(\ln(\rho)+a)\right)\right]f_\lambda(0)
\end{equation}
and whose covariance at any $x,y\in S_\lambda$ is by definition
\[
\E\left[\left(f_\lambda(x)-\E\left[f_\lambda(x)\, \big|\,f_\lambda(0)\right]\right)\left(f_\lambda(y)-\E\left[f_\lambda(y)\, \big|\,f_\lambda(0)\right]\right)\, \big|\,f_\lambda(0)\right]
\]
which equals
\begin{equation}\label{e.tent.prf.2}
\ln(\lambda)-\ln(\rho)+O(a)-\frac{\left(\ln(\lambda)-\ln(\rho)+O(a)\right)^2}{\ln(\lambda)+O(a)}= \ln(\rho)+O(\ln(\rho)/\ln(\lambda))+O(a)+O(1)\, .
\end{equation}
In particular, there exists $\rho_0=\rho_0(a)<+\infty$ such that if $\rho\geq \rho_0$, then, for each $\lambda\geq \rho\vee \lambda_1$, the vector $(f_\lambda(x))_{x\in S_\lambda}$ conditioned on $f_\lambda(0)$ is positively correlated. Take $\lambda\geq \lambda_0$. Let $\calH^\lambda$ be  the event that for all $x\in S_\lambda$, $f_\lambda(x)>0$. We want to find a lower bound for the probability of $\calH^\lambda$ conditioned on $f_\lambda(0)$, on the event that $f_\lambda(0)>-1$. To this end we start by proving the following estimate.
\begin{claim}\label{cl.tent.prf.1}
There exists $\rho_1=\rho_1(a,b,\alpha,n)<+\infty$ such that if $\rho\geq \rho_1$, for each $\lambda\geq \rho\vee\lambda_1$ and each $x\in S_\lambda$, on the event $f_\lambda(0)\geq -1$,
\[
\prob\left[\forall y\in S_\lambda \cap B(x,1/\lambda),\, f_\lambda(y)>0\, \big|\, f_\lambda(0)\right]\geq 1/3\, .
\]
\end{claim}
\begin{prf}
Fix $\rho\geq\rho_0$, $\lambda\geq\rho\vee\lambda_1$ and $x\in S_\lambda$. Set $V_\lambda\subset B(0,1)$ be the set of $z\in B(0,1)$ such that $x+z/\lambda\in S_\lambda$. For each $z\in B(0,1)$,
\begin{align*}
m_\lambda(z)&=\E\left[f_\lambda(x+z/\lambda)\, |\, f_\lambda(0)\right]\\
h_\lambda(z)&=f_\lambda(x+z/\lambda)-m_\lambda(z)\\
g_\lambda(z)&=h_\lambda(0)-h_\lambda(z)\, .
\end{align*}
Then, we have, for each $z\in B(0,1)$,
\begin{equation}\label{e.tent.claim.3}
f_\lambda(x+z/\lambda)=h_\lambda(0)-g_\lambda(z)+m_\lambda(z)\, .
\end{equation}
We will now show that, conditionally on $f_\lambda(0)$ and on the event $f_\lambda(0)>-1$, with positive probability, the three terms in the right-hand side of \eqref{e.tent.claim.3} satisfy inequalities that imply that $f_\lambda(x+z\lambda)>0$ for each $z\in V_\lambda$. Firstly, Equation \eqref{e.tent.prf.1} shows that on the event $f_\lambda(0)>-1$, $m_\lambda(z)$ is bounded from below uniformly for each $\lambda$ and each $z\in V_\lambda$. Let $m=m(a,n)>-\infty$ be a uniform lower bound. Next, conditionally on $f_\lambda(0)$, $g_\lambda$ is centered and for each $z,z'\in B(0,1)$,
\begin{align*}
\E\left[g_\lambda(z)^2\, |\, f_\lambda(0)\right]&\leq \E\left[g_\lambda(z)^2\right]\leq 2b^2|z|^\alpha\leq 2b^2\\
\E\left[(g_\lambda(z)-g_\lambda(z'))^2\, |\, f_\lambda(0)\right]&\leq4\E\left[(f_\lambda(x+z/\lambda)-f_\lambda(x+z'/\lambda))^2\right]\leq 4b^2 |z-z'|^\alpha\, .
\end{align*}
Here we used that, by the regression formula (Proposition 1.2 of \cite{azws}), variances do not increase under Gaussian conditioning. By Lemma \ref{l.tail_of_the_supremum} (with $V=B(0,1)$) there exist a constant $C=C(n,\alpha)<+\infty$ and a constant $u_0=u_0(b)\in]0,+\infty[$ such that for each $\lambda$,
\begin{equation}\label{e.tent.claim.1}
\prob\left[\sup_{B(0,1)}g_\lambda>Cb+u_0\, \Big|\, f_\lambda(0)\right]\leq 1/9\, .
\end{equation}
By equation \eqref{e.tent.prf.2}, uniformly in $\lambda\geq \rho$,
\[
\E\left[h_\lambda(0)^2\, |\, f_\lambda(0)\right]= \ln(\rho)+O(a)+O(1)\, .
\]
Also, conditionally on $f_\lambda(0)$, $h_\lambda(0)$ is centered. Hence, there exists $\rho_1=\rho_1\left(a,m,u_0,b,C\right)<+\infty$ such that if $\rho\geq \rho_1$,
\begin{equation}\label{e.tent.claim.2}
\prob\left[h_\lambda(0)>Cb+u_0-m\, |\, f_\lambda(0)\right]>4/9\, .
\end{equation}
Here $m$, $C$ and $u_0$ depend only on $a$, $b$, $\alpha$ and $n$ so $\rho_1$ depends only on $a$, $b$, $\alpha$ and $n$. Assume now that $f_\lambda(0)>-1$. By Equation \eqref{e.tent.claim.3}, since for each $z\in V_\lambda$ we have $m_\lambda(z)\geq m$, the event
\[
\left\{ h_\lambda(0)>Cb+u_0-m\right\}\cap\neg\left\{\sup_{B(0,1)}g_\lambda> Cb+u_0\right\}
\]
implies that $\forall y\in S_\lambda \cap B(x,1/\lambda)$, $f_\lambda(y)>0$. Hence, for each $\rho\geq \rho_1$, each $\lambda\geq\rho\vee\lambda_1$ and each $x\in S_\lambda$, we have shown that on the event $f_\lambda(0)>-1$,
\begin{align*}
\prob\left[\forall y\in S_\lambda\cap B(x,1/\lambda)\, f_\lambda(y)>0\, \big|\, f_\lambda(0)\right]\geq\prob&\left[h_\lambda(0)>Cb+u_0-m\, |\, f_\lambda(0)\right]\\
&-\prob\left[\sup_{B(0,1)}g_\lambda>Cb+u_0\, \Big|\, f_\lambda(0)\right]\\
\geq 4&/9-1/9\text{ by equations \eqref{e.tent.claim.1} and \eqref{e.tent.claim.2}}\\
\geq 1&/3\, .
\end{align*}
\end{prf}
From now on, we assume that $\rho\geq \rho_0\vee\rho_1$. Cover $S_\lambda$ with $N=N(\rho)$ balls $(B_i)_{1\leq i\leq N}$ of radius $1/\lambda$. Conditionally on $f_\lambda(0)$, for each $i\in\{1,\dots, N\}$, the event that the field $f_\lambda$ stays positive on $B_i$ is increasing. Therefore, by the FKG inequality (Lemma \ref{l.FKG}) we have, for each $\lambda\geq \rho\vee\lambda_1$,
\[
\prob\left[\calH^\lambda\, \big|\,  f_\lambda(0)\right]\geq \prod_{i=1}^N\prob\left[\forall x\in B_i\, f_\lambda(x)>0\, |\, f_\lambda(0)\right]\, .
\]
Now, by the previous claim, on the event that $f_\lambda(0)\geq -1$, the right hand side is greater than $3^{-N}$. Consequently, for each $\lambda\geq \rho\vee\lambda_1$,
\begin{align*}
\prob\left[\calC^\lambda\right]\geq \prob\left[f_\lambda(0)<0;\, \calH^\lambda\right]&\geq \E\left[\prob\left[\calH^\lambda\, \big|\, f_\lambda(0)\right]\one[-1\leq f_\lambda (0)<0]\right]\\
&\geq 3^{-N}\prob\left[-1\leq f_\lambda(0)<0\right]\, .
\end{align*}
Since $f_\lambda$ is centered and $\E[f_\lambda(0)^2]=\ln(\lambda)+O(a)$, there exist $\lambda_0=\lambda_0(a,\rho)<+\infty$ and $\kappa=\kappa(a,\rho)>0$ such that for each $\lambda\geq\lambda_0$,
\[
\prob\left[\calC^\lambda\right]\geq \kappa\left(\ln(\lambda)\right)^{-1/2}\, .
\]
Choosing $\rho=\rho_0\vee\rho_1$, one can assume that $\lambda_0$ and $\kappa$ depend only on $a$, $b$, $n$ and $\alpha$.
\end{proof}

\nocite{*}
\bibliography{bib_cfgfs}
\bibliographystyle{alpha}

%\ \\
\ \\
{\bf Alejandro Rivera}\\
Univ. Grenoble Alpes\\
UMR5582, Institut Fourier, 38610 Gi\`{e}res, France\\
alejandro.rivera@univ-grenoble-alpes.fr\\
Supported by the ERC grant Liko No 676999\\
 
\end{document}